\documentclass[review]{elsarticle}
\usepackage[usenames]{color}
\usepackage{extarrows}
\usepackage{hyperref}

\journal{Journal of \LaTeX\ Templates}

\usepackage{amsfonts}
\usepackage{amsmath, cases}
\usepackage{indentfirst}
\usepackage{graphicx}
\usepackage{latexsym,bm, amsthm}

\usepackage{epsfig}
\usepackage{latexsym}
\usepackage{amssymb}
\usepackage{CJK}
\usepackage{indentfirst}
\usepackage{amsmath}
\allowdisplaybreaks[2]%

\usepackage{xcolor}
\usepackage[normalem]{ulem} 
\newcommand\BJ{\bgroup\markoverwith
  {\textcolor{yellow}{\rule[-.5ex]{2pt}{3.5ex}}}\ULon}
\usepackage{geometry}
\newcommand{\rk}{{\rm rank}}
\newcommand{\R}{{\mathcal{R}}}
\newcommand{\Ind}{{\rm Ind}}

\newcommand{\M}{\widehat{A}}
\newcommand{\MP}{\widehat{A}^{\rm P}}
\newcommand{\PP}{{\rm P}}

\newcommand{\MD}{\widehat{A}^{ {+}}}

\begin{document}

\begin{frontmatter}

\title{ The Dual Index and Dual Core Generalized Inverse}

\author[1]{Hongxing Wang}
\ead{winghongxing0902@163.com}


\author[1]{Ju Gao}
 \ead{3160369785@qq.com}

\address[1]{School of Mathematics and Physics,
Guangxi Key Laboratory of Hybrid Computation and IC Design Analysis,
Guangxi University for Nationalities,
Nanning 530006, China}

\begin{abstract}
In this paper, we introduce the dual index and dual core generalized inverse (DCGI). By applying rank equation, generalized inverse and matrix decomposition, we give several characterizations of the dual  index when it is equal to one. And we get that if DCGI exists, then it is unique. We derive a compact formula for DCGI and a series of equivalent characterizations of the existence of the inverse.  It is worth nothing that the dual index of $\M$ is equal to one if and only if  its DCGI  exists. When the dual index of $\M$ is equal to one, we study  dual Moore-Penrose generalized inverse (DMPGI) and dual group generalized inverse (DGGI), and consider the relationships among DCGI,  DMPGI, DGGI, Moore-Penrose dual generalized inverse (MPDGI) and other dual generalized inverses. In addition, we consider symmetric dual matrix  and its dual generalized inverses.  At last, two examples are given to illustrate the application of DCGI in linear dual equations.
\end{abstract}


\end{frontmatter}

\section{Introduction}
\numberwithin{equation}{section}
\newtheorem{theorem}{T{\scriptsize HEOREM}}[section]
\newtheorem{lemma}[theorem]{L{\scriptsize  EMMA}}
\newtheorem{corollary}[theorem]{C{\scriptsize OROLLARY}}
\newtheorem{property}[theorem]{P{\scriptsize PROPERTY}}
\newtheorem{proposition}[theorem]{P{\scriptsize ROPOSITION}}
\newtheorem{remark}{R{\scriptsize  EMARK}}[section]
\newtheorem{definition}{D{\scriptsize  EFINITION}}[section]
\newtheorem{algorithm}{A{\scriptsize  LGORITHM}}[section]
\newtheorem{example}{E{\scriptsize  XAMPLE}}[section]
\newtheorem{problem}{P{\scriptsize  ROBLEM}}[section]
\newtheorem{assumption}{A{\scriptsize  SSUMPTION}}[section]

The concept of   dual number is introduced by Clifford in 1873
 and the name of the number is given by Study in 1903 \cite{Fischer2017}.
The dual number consists of a real unit 1 and a Clifford operator $\varepsilon$.
The dual number contains two real elements, i.e. $\widehat{a}=a+ \varepsilon a{'} $,
where the real elements $a$  and $a{'}$ are called the real part and dual part of $\widehat {a} $, respectively.
The rule is $\varepsilon\neq 0$, $0\varepsilon=\varepsilon 0=0$,
$1\varepsilon=\varepsilon1=\varepsilon$ and $\varepsilon^2=0$.
When we discuss the geometry of directed lines in space,
we can take the angle ``$\theta$" as the real part and the vertical distance ``$s$"
as the dual part to form the dual angle, $\widehat{\theta}=\theta+\varepsilon s$.
As an extension of the concept of dual number,
dual vector is to replace the real elements of a real vector with dual numbers
and is often used as mathematical expressions for helices.
A matrix with dual numbers as elements is called a dual matrix.
Denote an $m\times n$ dual matrix as $\M$
 that is written as the form
\begin{align}
\label{Def-MPDGI-0}
\M= A + \varepsilon B,
\end{align}
in which
$A\in {\mathbb{R}_{m, n}}$
and  $B\in {\mathbb{R}_{m, n}}$.
The matrix
 $A\in {\mathbb{R}_{m, n}}$($B\in {\mathbb{R}_{m, n}}$)
 is called the real(dual) part of the dual matrix $\M$.
The symbols
${\mathbb{D}}_{m, n}$
denotes the set of all ${m\times n}$ dual matrices;
$I_m$ is an  $m$-order identity matrix;
$\mathcal{R}\left({\M}\right)$ represents the range of  ${\M}$.
Furthermore, denote
$\M ^\mathrm{ T }=A^\mathrm{ T } + \varepsilon B^T$.
When ${\M}^\mathrm{T}={\M}$,
we say that  ${\M}$ is symmetric.

Dual matrices are used in many fields today.
In Kinematics, for example, with the aid of the principle of transference \cite{d17g},
many problems can be initially stated under the condition of spherical motion
 and then extended to spiral motion after the dualization of the equation,
which makes the dual matrix widely used
in space agency kinematics analysis and synthesis \cite{d1g, d3g, d2g, d5g}
and robotics \cite{d10g, d8g, d7g, d6g, d11g}.
Their presence is also felt in other areas of science and engineering,
which has raised interest
in various aspects of linear algebra and computational methods associated
with their use \cite{d13g, d15g, d14g, d12g}.
The pioneering work in the engineering applications of dual algebra is
 made by Keler\cite{d19g}, Beyer\cite{d20g}, etc.
\

In \cite{d11g}, Pennestr\`{\i} and Valentini
introduce {the}
 \emph{Moore-Penrose dual generalized inverse} (MPDGI):
let  $\M= A + \varepsilon B $,
then the MPDGI of $\M$ is denoted by $\MP$ and is displayed in the form
$ \MP= A^{+} - \varepsilon A^{+} BA^{+}$.
For a given dual matrix $\M$,
if there exists a dual matrix $\widehat{X}$
 satisfying
$$ (1)\ \M\widehat{X}\M = \M ,
  \  (2) \ \widehat{X}\M\widehat{X} = \widehat{X} ,
  \  (3) \ \left( {\M\widehat{X}} \right)^\mathrm{ T } = \M\widehat{X} ,
  \  (4) \ \left( {\widehat{X}\M} \right)^\mathrm{ T } = \widehat{X}\M,$$
then we call $\widehat{X}$ {the}   \emph{dual Moore-Penrose generalized inverse} (DMPGI) of $\M$,
and
 denote it by $\MD$ \cite{d12g}.
It is worth noting that for any dual matrix, its MPDGI always exists, while its DMPGI may not exist.
Furthermore,
if $\widehat{X}$ satisfies $\M\widehat{X}\M = \M$,
we call $\widehat{X}$ a $\{1\}$-dual generalized inverse of ${\M}$,
and denote it ${\M}^{\{1\}}$;
 if $\widehat{X}$ satisfies $\M\widehat{X}\M = \M$
 and
 $\left( {\M\widehat{X}} \right)^\mathrm{ T } = \M\widehat{X}$,
we call $\widehat{X}$ a $\{1, 3\}$-dual generalized inverse of ${\M}$,
and denote it  as ${\M}^{\{1, 3\}}$.
Recently,  Wang has given some necessary and sufficient conditions for a dual matrix to have the DMPGI,
and
 some equivalent relations between DMPGI and MPDGI in \cite{d21g}.

\begin{lemma}[\cite{d21g}]
\label{lemma1}
Let ${\M}=A+\varepsilon B\in {\mathbb{D}_{m, n}}$, then the following conditions are equivalent:

{\bf (i).}  The DMPGI ${\M}^{+}$ of ${\M}$ exists;

{\bf (ii).} $\left(I_m-AA^{+}\right)  B(I_n-A^{+}A)=0$;

{\bf (iii).} $\rk  \left(\left[\begin{matrix}
       B&       A
\\     A&         0
\end{matrix}\right]\right)=2\rk (A)$.

Furthermore, when the DMPGI ${\M}^{+}$ of ${\M}$ exists,
\begin{align}
\label{The-DMPGI}
{\M}^{+}
=
A^{+}
-\varepsilon
 \left(A^{+}BA^{+}-\left(A^\mathrm{ T }A\right)^{+}B^\mathrm{ T }\left(I_m-AA^{+}\right)-\left(I_n-A^{+}A\right)B^\mathrm{ T }\left(AA^\mathrm{ T }\right)^{+}\right).
\end{align}
\end{lemma}

 \begin{lemma}[\cite{d21g}]
\label{lemma2}
 Let ${\M}=A+\varepsilon B\in{\mathbb{D}}_{m, n}$,  then
 the  DMPGI ${\M}^{+}$ of ${\M}$ exists, and ${\M}^{+}={\M}^{\PP}$
 if and only if
$\left(I_m-AA^{+}\right)B=0$  and $B\left(I_n-A^{+}A\right)=0.$
\end{lemma}
MPDGI and DMPGI are used in many aspects.
For example, in \cite{d19g},
Udwadia and Pennestr\`{\i} apply MPDGI  to various motions such as
 rigid body translation and dual angular velocity acquisition.
In \cite{d12g}, DMPGI is applied to kinematic synthesis of spatial mechanisms,
and a series of numerical examples about calculation and application of DMPGI
to kinematic synthesis of linkage mechanisms are given.
In addition,
MPDGI are also used in many inverse problems of kinematics and analysis of  machines and mechanisms in \cite{d1g}.

Next, Jin and Zhang\cite{Zhong2022AIMS}   introduce   the  dual group generalized inverse (DGGI):
let ${\M} $ be an $n$-square dual matrix.
If there exists  an $n$-square dual matrix $\widehat{G}$  satisfying
\begin{align*}
 \
 (1)\ {\M}\widehat{G}{\M}={\M},
  \quad
  \
  (2)\ \widehat{G}{\M}\widehat{G}=\widehat{G},
   \quad
   \
   {  (5)}\  {\M}\widehat{G}=\widehat{G}{\M},
\end{align*}
then ${\M}$ is called {  a dual group generalized invertible  matrix},
and $\widehat{G}$ is the  DGGI of ${\M}$ , which is recorded as ${\M}^\#$.
 Jin and Zhang\cite{Zhong2022AIMS}
 give some necessary and sufficient conditions for a dual matrix to have DGGI
 and
 apply DGGI to study linear dual equations.
 \begin{lemma}[\cite{Zhong2022AIMS}]
\label{core-yingli}
 Let ${\M}=A+\varepsilon B$ be a dual matrix
 with $A, B\in {\mathbb{R}_{n, n}}$
  and ${\rm Ind}\left(A\right)=1$, then
the DGGI of  ${\M}$ exists if and only if $\left(I_n-AA^{\#}\right)  B(I_n-AA^{\#})=0$.

Furthermore, if the dual group inverse of ${\M}$ exists,
then
\begin{align}
\label{qunni-gongshi}
{\M}^{\#}=A^{\#}+\varepsilon R,
\end{align}
where
\begin{align}
\label{The-group}
 R=-A^{\#}BA^{\#}+\left(A^{\#}\right)^2B\left(I_n-AA^{\#}\right)+\left(I_n-AA^{\#}\right)B\left(A^{\#}\right)^2.
\end{align}
\end{lemma}

 Most application of dual algebra
 in Kinematics require numerical solutions to linear dual equations.
 As showed in \cite{Udwadia2021mamt156},
  Udwadia introduces  the norm
   of   dual vector and uses some properties of dual generalized inverses in solving linear dual equations.
 Various dual generalized inverses are useful for solving
 consistent linear dual equations or
  inconsistent lineardual dual equations.

  In particular, in \cite{dQIg}, both a total order and an
absolut value function for dual numbers are put forward by Qi, Ling and Yan.
 Then they give the definition of the magnitude of a dual quaternion as a dual number. Furthermore, 1-norm, $\infty$-norm and 2-norm are extended to dual quaternion vectors in their article.

\

 It is known that Moore-Penrose inverse and group inverse belong to generalized inverses in complex fields.
Other well-known generalized inverses are Drazin inverse, core-EP inverse, and so on.
The core inverse means that when the index of  $A\in {\mathbb{R}_{n, n}}$ is 1,
there is a unique matrix $X\in {\mathbb{R}_{n, n}}$,
which satisfies $AXA =A$, $AX^{2}= X$ and $\left(AX\right)^\mathrm{ T }=AX$.
We call it the core inverse of matrix $A$, expressed by $A^{\mbox{\tiny\textcircled{\#}}}$.
 In \cite{Baksalary2010LMA681}, Baksalary and   Trenkler get $A^{\mbox{\tiny\textcircled{\#}}}=A^{\#}AA^+$.
The core inverse has good properties.
It can be used to solve many problems, especially in  constrained least squares problem.
 Although the above literatures have discussed a series of related problems of dual generalized inverses,
the dual core generalized inverse of a dual matrix has not been systematically studied,
and related problems of the dual index have not been discussed.
On the basis of the above researches, the concepts of the dual  index
and dual core generalized inverse (DCGI) are introduced in this paper.
Furthermore,
 it is proved that when the dual index of dual matrix is 1, there must be the DCGI of the dual matrix.
The sufficient and necessary condition, namely the index is 1, is used to identify the existence of DCGI,
which makes the problem more concise and clear.
At the same time, we also give other equivalent conditions for the existence of  DCGI
and the compact formula for  DCGI,
as well as the relations among  DCGI,  DGGI,  DMPGI and  MPDGI of the dual matrix,
and discuss some special dual matrices.
At last, we solve two linear dual equations by applying DCGI.

\section{ Dual Index One}
\label{Section-DualIndex}

 In complex(real)  field,
the index is necessary for studying generalized inverse  and its related problems.
For example,
it is known  that if and only if the group (core) inverse of a matrix exist, the index of the matrix is equal to one.
Moreover,
Wei et al \cite{Diao2000jcam-1,Lin2009laa1665} consider singular linear structured system with index one.
 In this section, we introduce the dual index of a dual matrix.
 When the dual index is equal to one,
  we give its some characterizations.
  Furthermore,
by using the dual index,
we  study the  dual group generalized invertible  matrix.

\begin{definition}
\label{Def-Dual-Index-1}
 Let ${\M}$   be   an $n$-square   dual   matrix.
 If $\R\left({{\M}}^{2}\right)=\R\left({\M}\right)$,
 then the dual index of ${\M}$ is equal to one,
 and it is recorded as $\Ind\left({\M}\right)=1$.
\end{definition}

Suppose that $\R\left({{\M}}^{2}\right) = {\R\left({\M}\right)}$,
where ${\M}=A+\varepsilon B$, $A$ and $B$
are $n\times n$ real matrices.
To obtain $\R\left({\widehat{A}}^{2}\right)=\R\left(\widehat{A}\right)$,
 we   need to prove $\R\left(\widehat{A}\right)\subseteq{\R\left(\widehat{A}^{2}\right)}$ and ${ \R\left(\widehat{A}^{2}\right)}\subseteq{\R\left(\widehat{A}\right)}$.
 It is   obvious that  the latter is true.
 Next, we {  need to prove} $\R\left(\widehat{A}\right)\subseteq{\R\left(\widehat{A}^{2}\right)}$,
 that is, there exists $\widehat{X}=X_1+\varepsilon X_2$,
  which makes
 \begin{align}
 \nonumber
 {\M}^{2}\widehat{X}={\M}.
 \end{align}
Put ${\M}=A+\varepsilon B$ and $\widehat{X}=X_1+\varepsilon X_2$ into the above equation to get
 \begin{align}
 \nonumber
  A+\varepsilon B
  =
  \left(A+\varepsilon B\right)^2\left({X}_1+\varepsilon{X}_2\right)
  =
  A^2{X}_1+\varepsilon\left(A^2{X}_2+\left(AB+BA\right){X}_1\right),
 \end{align}
that is,
\begin{subnumcases}{}
 \label{Sec-2-2.2-1-a}
 A^2{X}_1=A,
 \\
 \label{Sec-2-2.2-1-b}
 A^2{X}_2+\left(AB+BA\right){X}_1=B.
 \end{subnumcases}

From (\ref{Sec-2-2.2-1-a}),
we see that
 $A^2{X}_1=A$ is consistent
 if and only if
\begin{align}
 \label{Sec-2-2.2-3}
 \rk \left(A\right)=\rk \left(A^2\right),
\end{align}
 that is, the index of real matrix $A$ is equal to 1.
And we can obtain the  general solution to (\ref{Sec-2-2.2-1-a})
\begin{align}
 \label{Sec-2-2.2-2}
{X}_1=A^{\# }+\left(I_n-A^{\# }A\right)Y,
\end{align}
where $Y$ is arbitrary.

By substituting  (\ref{Sec-2-2.2-2}) into (\ref{Sec-2-2.2-1-b}),
 we get $B =A^2X_2+\left(AB+BA\right)\left(A^{\#}+\left(I_n-A^{\#}A\right)Y\right)$.
 From
$A\left(A^{\#}+\left(I_n-A^{\#}A\right)Y\right)=AA^{\#}$,
it follows that
$B
=A^2X_2+AB\left(A^{\#}+\left(I_n-A^{\#}A\right)Y\right)+{BAA^{\#}}
=A^2X_2+ABA^{\#}+AB\left(I_n-A^{\#}A\right)Y+{  BAA^{\#}} $.
Therefore,
\begin{align}
 \label{Sec-2-2.2-5}
\left[\begin{matrix}A^2&AB\left(I_n-A^{\#}A\right)\end{matrix}\right]
\left[\begin{matrix} X_2\\Y\end{matrix}\right]
=B-ABA^{\#}{  -BAA^{\#}}.
\end{align}
If the above equation is consistent, which is equivalent to
\begin{align}
 \label{Sec-2-2.2-4}
 \rk \left(\left[\begin{matrix}A^2&AB\left(I_n-A^{\#}A\right)&B-ABA^{\#}{  -BAA^{\#}}\end{matrix}\right]\right)
=\rk \left(\left[\begin{matrix} A^2&AB\left(I_n-A^{\#}A\right)\end{matrix}\right]\right).
\end{align}
Since
 \begin{align*}
\left[\begin{matrix}A^{2}&AB\left(I_n-A^{\# }A\right)\end{matrix}\right]
\left[\begin{matrix}
          I_n&        -A^{\# }B\left(I_n-A^{\# }A\right)
\\        0 &        I_n
\end{matrix}\right]
=\left[\begin{matrix}A^{2}& 0\end{matrix}\right],
\end{align*}
by applying   (\ref{Sec-2-2.2-3}) we  obtain
\begin{align*}
\nonumber\rk \left(A^{2}\right)=\rk \left(A\right)
&=
\rk \left(\left[\begin{matrix}A^{2} & AB\left(I_n-A^{\# }A\right) &  B-\left(AB+BA\right)A^{\#}\end{matrix}\right]\right).
\\
&=
\rk \left(\left[\begin{matrix}A &B-\left(AB+BA\right)A^{\#}\end{matrix}\right]\right)
\\
&=
\rk \left(\left[\begin{matrix}A&B\left(I_n-AA^{\#}\right)\end{matrix}\right]\right).
\end{align*}
Then the consistency of (\ref{Sec-2-2.2-5}) is equivalent to
\begin{align}
\label{Sec-2-2.2-6}
\rk \left(A\right)
=\rk \left(\left[\begin{matrix}A & B(I_n-AA^{\#})\end{matrix}\right]\right).
\end{align}

According to (\ref{Sec-2-2.2-3}) and (\ref{Sec-2-2.2-6}),
 the dual index of ${\M}$ is equal to 1,
which is equivalent to
\begin{align}
\label{g 5}
\rk \left(A^{2}\right)
=\rk (A)
=\rk \left(\left[\begin{matrix}A & B\left(I_n-AA^{\#}\right)\end{matrix}\right]\right).
\end{align}
Therefore,
we have the following theorem.
\begin{theorem}
\label{The-2-4}
Let ${\M}= A+\varepsilon B\in{\mathbb{D}}_{n, n}$ and $\rk \left(A\right)= r$,
then the dual index of ${\M}$ is equal to one,
which is equivalent to   ${\rm Ind}\left(A\right)=1$, and
\begin{align}
\label{laojia}
\rk \left(A\right)
=
\rk\left (\left[\begin{matrix}A&B(I_n-AA^{\#}\end{matrix}\right]\right).
\end{align}
\end{theorem}

\

Next,   we  present one well-known matrix decomposition\cite[Corollary 6]{Hartwig1984LMA241}
and several corresponding decompositions of generalized inverses, which will be used
in the following part of the paper.
Let  $A\in \mathbb{R}_{n, n}$ with  $\rk \left(A\right)= r$,
then
\begin{align}
\label{core-1}
A=U\left[\begin{matrix}
     \Sigma K&         \Sigma L
\\        0&         0
\end{matrix}\right]U^\mathrm{ T }
\end{align}
where $U\in \mathbb{R}_{n, n}$ is unitary,
 nonsingular   $ \Sigma=diag(\sigma_{1}, \cdots, \sigma_{r})$  is the diagonal matrix of singular values of $A$,
$\sigma_{1}\geq \cdots \geq \sigma_{r}> 0$,
and $K\in \mathbb{R}_{r, r}$, $L\in \mathbb{R}_{r, n-r}$ satisfy
\begin{align}
\label{core-2}
KK^\mathrm{ T }+LL^\mathrm{ T }=I_r.
\end{align}

By applying the decomposition,
Baksalary and Trenkler\cite{Baksalary2010LMA681}
get
\begin{align}
\label{core-3}
A^+=U\left[\begin{matrix}
     K^\mathrm{ T }\Sigma^{-1} &         0
\\   L^\mathrm{ T }\Sigma^{-1} &         0
\end{matrix}\right]U^\mathrm{ T }.
\end{align}
Especially,
when the index of $A$ is one,
the necessary and sufficient condition for the existence of $A^{\mbox{\tiny\textcircled{\#}}}$
is that $K$ is nonsingular.
In \cite{Baksalary2010LMA681}, by applying (\ref{core-1}) Baksalary and Trenkler also  give characterizations of core inverse and group inverse:
\begin{align}
\label{core-5}
A^{\mbox{\tiny\textcircled{\#}}}
&=A^{\#}AA^+
\\
\label{core-6}
&=
U\left[\begin{matrix}
      (\Sigma K)^{-1}&     0
\\                       0&     0
\end{matrix}\right]U^\mathrm{ T },
\\
\label{core-4}
A^{\# }&=U\left[\begin{matrix}
      K^{-1}\Sigma^{-1}&     K^{-1}\Sigma^{-1}K^{-1}L
\\                    0&         0
\end{matrix}\right]U^\mathrm{ T }.
\end{align}

 Based on the premise that the index of  $A\in \mathbb{R}_{n, n}$ is 1,
 we analyze (\ref{laojia}).
Let the decomposition of $A$ be of the form in  (\ref{core-1}).
Write
\begin{align}
\label{core-7}
B=U\left[\begin{matrix}
     B_1&         B_2
\\   B_3&         B_4
\end{matrix}\right]U^\mathrm{ T },
\end{align}
 where $B_1$ is an $r$-square matrix and $r=\rk(A)$.
 By substituting  (\ref{core-1}), (\ref{core-4}) and (\ref{core-7}) to $B\left(I_n-AA^{\#}\right)$,
we get
\begin{align*}
B\left(I_n-AA^{\#}\right)=U
\left[\begin{matrix}
     0&        -B_1K^{-1}L+B_2
\\   0&        -B_3K^{-1}L+B_4
\end{matrix}\right]U^\mathrm{ T }.
\end{align*}
It follows from   (\ref{laojia})  and (\ref{core-1}) that
\begin{align*}
\rk\left(
\left[\begin{matrix}
     \Sigma K&   \Sigma L
\\          0&        0
\end{matrix}\right]\right)=\rk\left(
\left[\begin{matrix}
      \Sigma K&        \Sigma L&     0&       -B_1K^{-1}L+B_2
\\           0&               0&     0&       -B_3K^{-1}L+B_4
\end{matrix}\right]\right),
\end{align*}
 which
implies that
$B_4=B_3K^{-1}L$.

In summary,
we get that
$\rk (A)=\rk\left(\left [\begin{matrix}A &  B(I_n-AA^{\#})\end{matrix}\right]\right)$
  if and only if
$B_4=B_3K^{-1}L$.
Therefore, we have the following Theorem \ref{The-2-40}.
\begin{theorem}
\label{The-2-40}
Let  ${\M}= A+\varepsilon B\in{\mathbb{D}}_{n, n}$, $\rk \left(A\right)= r$,
 $A$ and $B$ have the forms as  in $(\ref{core-1})$ and  $(\ref{core-7})$,
respectively.
Then
the dual index of ${\M}$ is equal to one, which is equivalent to
   ${\rm Ind}\left(A\right)=1$
 and $B_4=B_3K^{-1}L$.
\end{theorem}

Furthermore,
since ${\rm Ind}\left(A\right)=1$,
applying (\ref{core-1}) and (\ref{core-7}),
 it is easy to check that
\begin{align*}
\rk \left(\left[\begin{matrix}
       B&       A
\\     A&       0
\end{matrix}\right]\right)&=
\rk \left(\left[\begin{matrix}
           B_1&             B_2&      \Sigma K&    \Sigma L
\\         B_3&             B_4&             0&       0
\\    \Sigma K&        \Sigma L&             0&       0
\\            0&              0&             0&       0
\end{matrix}\right]\right)
\nonumber
\\
&=2\rk \left(A\right) + \rk \left(B_4-B_3K^{-1}L\right).
\end{align*}
By applying Theorem \ref{The-2-40}, we have the following Theorem \ref{dingli1}.

\begin{theorem}
\label{dingli1}
Let ${\M}=A+\varepsilon B\in{\mathbb{D}}_{n, n}$ and $\rk\left(A\right)=r$,
then the dual index of ${\M}$ is equal to one,
which
 is equivalent to ${\rm Ind}\left(A\right)=1$,
   and
\begin{align*}
\rk \left(\left[\begin{matrix}
       B&       A
\\     A&       0
\end{matrix}\right]\right)=2\rk (A).
\end{align*}
\end{theorem}

 In the following theorems, we give some equivalent characterizations with dual index one.
\begin{theorem}
\label{zhibiaoMP}
Let ${\M}=A+\varepsilon B\in{\mathbb{D}}_{n, n}$,
where $\rk(A)=r$,
then the dual index of ${\M}$ is one
if and only if
${\rm Ind}\left(A\right)=1$  and $\left(I_n-AA^{+}\right)B\left(I_n-A^{+}A\right)=0$.
\end{theorem}

\begin{proof}
  By applying  equations (\ref{core-1}), (\ref{core-3}) and (\ref{core-7}),
we get
\begin{align*}
 I_n-AA^{+} &=
U\left[\begin{matrix}
     0&         0
\\   0&         I_{n-r}
\end{matrix}\right]U^\mathrm{ T },
\\
 I_n-A^{+}A &=
U\left[\begin{matrix}
     I_r-K^\mathrm{ T }K&                   -K^\mathrm{ T }L
\\    -L^\mathrm{ T }K&         I_{n-r}-L^\mathrm{ T }L
\end{matrix}\right]U^\mathrm{ T }.
\end{align*}
 Then
\begin{align}
\nonumber
\left(I_n-AA^{+}\right)&B\left(I_n-A^{+}A\right)
\\
\label{xuyaoyong}
&=U\left[\begin{matrix}
                                             0&                      0
\\   B_3-B_3K^\mathrm{ T }K-B_4L^\mathrm{ T }K&          -B_3K^\mathrm{ T }L+B_4-B_4L^\mathrm{ T }L
\end{matrix}\right]U^\mathrm{ T }.
\end{align}

Let the dual index of ${\M}$ be  equal to one.
According to Theorem \ref{The-2-40}, the index of $A$ is one and $B_4=B_3K^{-1}L$.
Substituting   $B_4=B_3K^{-1}L$
 into $B_3-B_3K^\mathrm{ T }K-B_4L^\mathrm{ T }K$
 and
 $ -B_3K^\mathrm{ T }L+B_4-B_4L^\mathrm{ T }L$,
applying (\ref{core-2}) we get
\begin{align*}
B_3-B_3K^\mathrm{ T }K-B_4L^\mathrm{ T }K
&=B_3-B_3K^\mathrm{ T }K-B_3K^{-1}LL^\mathrm{ T }K
\\
&=B_3-B_3K^\mathrm{ T }K-B_3K^{-1}K+B_3K^{-1}KK^\mathrm{ T }K
\\
&=B_3-B_3K^\mathrm{ T }K-B_3+B_3K^\mathrm{ T }K=0
\end{align*}
and
\begin{align*}
-B_3K^\mathrm{ T }L+B_4-B_4L^\mathrm{ T }L
&=-B_3K^\mathrm{ T }L+B_3K^{-1}L-B_3K^{-1}LL^\mathrm{ T }L
\\
&=-B_3K^\mathrm{ T }L+B_3K^{-1}L-B_3K^{-1}L+B_3K^{-1}KK^\mathrm{ T }L
\\
&=-B_3K^\mathrm{ T }L+B_3K^{-1}L-B_3K^{-1}L+B_3K^\mathrm{ T }L=0.
\end{align*}
Therefore,
from (\ref{xuyaoyong}),
it follows that  $\left(I_n-AA^{+}\right)B\left(I_n-A^{+}A\right)=0$.

\

Conversely,
let the index of $A$ is one and $\left(I_n-AA^{+}\right)B\left(I_n-A^{+}A\right)=0$.
Applying (\ref{xuyaoyong}) gives
\begin{subnumcases}{}
 \nonumber
B_3-B_3K^\mathrm{ T }K-B_4L^\mathrm{ T }K=0,
 \\
 \nonumber
-B_3K^\mathrm{ T }L+B_4-B_4L^\mathrm{ T }L=0,
 \end{subnumcases}
that is,
\begin{subnumcases}{}
 \label{0-1}
B_3-B_3K^\mathrm{ T }K=B_4L^\mathrm{ T }K,
 \\
 \label{0-2}
B_3K^\mathrm{ T }L=B_4-B_4L^\mathrm{ T }L.
 \end{subnumcases}

 Since the index of $A$ is 1, it is known that $K$ is a singular matrix.
 Post-multiplying both sides of (\ref{0-1}) by $K^{-1}L$,
  we get
\begin{align}
\label{huajianxiaogongshi}
B_3K^{-1}L-B_3K^\mathrm{ T }L=B_4L^\mathrm{ T }L.
\end{align}
  By substituting equation (\ref{0-2}) into (\ref{huajianxiaogongshi}),
  we get
$B_3K^{-1}L-B_4+B_4L^\mathrm{ T }L=B_4L^\mathrm{ T }L$,
 that is, $B_4=B_3K^{-1}L$.
In summary, the index of $A$ is one and $B_4=B_3K^{-1}L$.
It follows from Theorem \ref{The-2-40} that the dual index of ${\M}$ is one.
\end{proof}

\begin{theorem}
\label{zhibiaoqunni}
Let ${\M}=A+\varepsilon B\in{\mathbb{D}}_{n, n}$,
where $\rk(A)=r$,
then the dual index of ${\M}$ is one
if and only if
${\rm Ind}\left(A\right)=1$
and
$\left(I-AA^{\#}\right)B\left(I-AA^{\#}\right)=0$.
\end{theorem}

\begin{proof}
Let  ${\rm Ind}\left(A\right)=1$.
By applying equations (\ref{core-1}), (\ref{core-4}) and (\ref{core-7}),
we can get
\begin{align}
\nonumber
\left(I_n-AA^{\#}\right) &=
U\left[\begin{matrix}
     0&         -K^{-1}L
\\   0&         I_{n-r}
\end{matrix}\right]U^\mathrm{ T },
\\
\label{yong}
\left(I_n-AA^{\#}\right)B\left(I_n-AA^{\#}\right) & =
U\left[\begin{matrix}
     0&         K^{-1}LB_3K^{-1}L-K^{-1}LB_4
\\   0&                      -B_3K^{-1}L+B_4
\end{matrix}\right]U^\mathrm{ T }.
\end{align}

If the dual index of ${\M}$ is one,
from Theorem \ref{The-2-40},
 we can get the index of $A$ is one and
$B_4=B_3K^{-1}L$.
 Therefore,
$K^{-1}LB_3K^{-1}L-K^{-1}LB_4=0$ and $ -B_3K^{-1}L+B_4=0$.
It follows from (\ref{yong}) that
 $\left(I-AA^{\#}\right)B\left(I-AA^{\#}\right)=0$.

Conversely,
let the index of $A$ is one and $\left(I-AA^{\#}\right)B\left(I-AA^{\#}\right)=0$.
Applying  (\ref{yong}) gives $B_4=B_3K^{-1}L$.
To sum up, the index of $A$ is one and $B_4=B_3K^{-1}L$.
 Furthermore, according to Theorem \ref{The-2-40},
 we get that the dual index of ${\M}$ is one.
\end{proof}

 By applying Lemma \ref{core-yingli} and  Theorem \ref{zhibiaoqunni},
  we get the following Theorem \ref{zhibiaoqunni-xin} that discusses the relationship between DGGI and dual index one.

\begin{theorem}
\label{zhibiaoqunni-xin}
Let ${\M}=A+\varepsilon B\in  {\mathbb{D}}_{n, n}$,
then the dual index of ${\M}$ is one
if and only if
${\M}^{\#}$ exists.
\end{theorem}

\begin{proof}
From  Lemma \ref{core-yingli},
we see that  DGGI of  ${\M}$ exists if and only if
${\rm Ind}\left(A\right)=1$
and
$\left(I_n-AA^{\#}\right)  B(I_n-AA^{\#})=0$.
From  Theorem \ref{zhibiaoqunni},
we see that
the dual index of ${\M}$ is one
if and only if
${\rm Ind}\left(A\right)=1$
and
$\left(I-AA^{\#}\right)B\left(I-AA^{\#}\right)=0$.
 Therefore,
 we get that   the dual index of ${\M}$ is one
if and only if
${\M}^{\#}$ exists.
\end{proof}

By applying Theorem \ref{zhibiaoMP},
we see that
 the dual index of ${\M}$ is one
if and only if
${\rm Ind}\left(A\right)=1$  and $\left(I_n-AA^{+}\right)B\left(I_n-A^{+}A\right)=0$.
By applying Lemma \ref{lemma1},
we see that
the DMPGI ${\M}^{+}$ of ${\M}$ exists
if and only if
 $\left(I_m-AA^{+}\right)  B(I_n-A^{+}A)=0$.
 Therefore, we get 
 the relationship between dual index one and DMPGI
  in the following Theorem \ref{zhibiaoMPni}.

\begin{theorem}
\label{zhibiaoMPni}
Let ${\M}=A+\varepsilon B\in  {\mathbb{D}}_{n, n}$,
 then the dual index of ${\M}$ is one if and only if
  ${\rm Ind}\left(A\right)=1$ and ${\M}^{+}$ exists.
\end{theorem}

\begin{theorem}
\label{tianjiatui}
Let ${\M}=A+\varepsilon B\in  {\mathbb{D}}_{n, n}$
and ${\rm Ind}\left(A\right)=1$ , then ${\M}^{+}$ exists if and only if
$\rk \left(A\right)
=\rk\left (\left[\begin{matrix}A&B\left(I_n-AA^{\#}\right)\end{matrix}\right]\right)$.
\end{theorem}

\begin{proof}

``$\Longrightarrow$''
\
If ${\M}^{+}$ exists,
then $\rk  \left(\left[\begin{matrix}
       B&       A
\\     A&         0
\end{matrix}\right]\right)=2\rk (A)$ is known by the Lemma\ref{lemma1}.
Therefore,
 when the index of $A$ is one, the dual index of ${\M}$ is equal to one from Theorem \ref{dingli1}.
 It is also known
 $\rk (A)=\rk\left (\left[\begin{matrix}A&B(I_n-AA^{\#}\end{matrix}\right]\right)$
 from Theorem \ref{The-2-4}.

``$\Longleftarrow$''
\
 Let $\rk (A)=\rk\left (\left[\begin{matrix}A&B(I_n-AA^{\#}\end{matrix}\right]\right)$.
 When the index of $A$ is one, the dual index of ${\M}$ is equal to one from Theorem \ref{The-2-4}.
 So $\rk  \left(\left[\begin{matrix}
       B&       A
\\     A&         0
\end{matrix}\right]\right)=2\rk (A)$ from Theorem \ref{dingli1}.
And by (i), (iii) of Lemma \ref{lemma1}, we know that ${\M}^{+}$ exists.
\end{proof}

\section{Dual Core Generalized Inverse}
\label{Section-DCGI}
It is well known that
one matrix is group invertible
if and only if
its index is one
if and only if
it is core invertible
in ${\mathbb{R}_{n, n}}$.
In Section \ref{Section-DualIndex},
we
get that  the dual index of ${\M}=A+\varepsilon B$ is one
if and only if
${\M}^{\#}$ exists.
In this section,
we introduce DCGI,
 give some properties and characterizations of the inverse,
and
consider relationships among  DCGI, DGGI and  dual index one.
 Meanwhile,
 we also give  characterizations of some other interesting dual generalized inverses.

\subsection{Definition and uniqueness  of DCGI}
\begin{definition}
\label{Def-2-1}

Let ${\M} $ be an $n$-square dual matrix.
If there exists  an $n$-square dual matrix $\widehat{G}$  satisfying
\begin{align}
\label{Def-2-1-Eq-1}
 (1)\
 {\M}\widehat{G}{\M}={\M},
  \quad
   \left(2'\right) \
  {\M}\widehat{G}^2=\widehat{G},
   \quad
    (3)\
    \left({\M}\widehat{G}\right)^\mathrm{ T }={\M}\widehat{G},
\end{align}
then ${\M}$ is called a dual core generalized invertible  matrix,
and $\widehat{G}$ is the dual core generalized inverse (DCGI) of ${\M}$,
which is recorded as  ${\M}^{\mbox{\tiny\textcircled{\#}}}$.
\end{definition}

\begin{theorem}
\label{The-2-1}
Let ${\M}=A+\varepsilon B\in  {\mathbb{D}}_{n, n}$,
then the existence of the DCGI of ${\M}$ is equivalent to the existence of $G$ and $R$,
which meets the following requirements:
$G=A^{\mbox{\tiny\textcircled{\#}}}$
and
 \begin{subnumcases}{}
\label{The-2-1-Eq-1}
BGA+ARA+AGB =B,
\\
\label{The-2-1-Eq-2}
AGR+ARG+BG^2 =R,
\\
\label{The-2-1-Eq-3}
\left(AR+BG\right)^\mathrm{ T }=AR +BG.
 \end{subnumcases}
Furthermore,
 $\widehat{G}=G+\varepsilon R$ is the DCGI of ${\M}$.
\end{theorem}

\begin{proof}
From ${\M}=A+\varepsilon B$,
 $\widehat{G}=G+\varepsilon R$
 and
\begin{align}
\nonumber
\left\{
\begin{array}{l}
{\M}\widehat{G}{\M}
=
 \left(A+\varepsilon B\right)\left(G+\varepsilon R\right)\left(A+\varepsilon B\right)
=AGA+\varepsilon\left(BGA+ARA+AGB\right),
\\
\nonumber
{\M}\widehat{G}^2
=
 \left(A+\varepsilon B\right)\left(G+\varepsilon R\right)^2
=AG^2+\varepsilon\left(AGR+ARG+BG^2\right),
\\
\nonumber
({\M}\widehat{G})^\mathrm{ T }
=\left(\left(A+\varepsilon B\right)\left(G+\varepsilon R\right)\right)^\mathrm{ T }
=(AG)^\mathrm{ T } +\varepsilon\left(AR+BG\right)^\mathrm{ T },
\end{array}\right.
\end{align}
we get that
${\M}\widehat{G}{\M}={\M}$,
 $ {\M}\widehat{G}^2=\widehat{G}$
and
    $({\M}\widehat{G})^\mathrm{ T }={\M}\widehat{G}$ are respectively equivalent to
\begin{align}
\nonumber
\left\{
\begin{array}{l}
AGA=A,
\
BGA+ARA+AGB=B,
\\
AG^2=G,
\
AGR+ARG+BG^2=R,
\\
\left(AG\right)^\mathrm{ T }=AG,\
\left(AR+BG\right)^\mathrm{ T }=AR+BG.
\end{array}\right.
\end{align}
  Since $AGA=A$, $AG^2=G$ and $\left(AG\right)^\mathrm{ T }=AG$,
  we have  $G=A^{\mbox{\tiny\textcircled{\#}}}$.
  Therefore, if the DCGI of ${\M}$ exists,
  and $\widehat{G}=G +\varepsilon R$ is the DCGI of ${\M}$,
 then $G =A^{\mbox{\tiny\textcircled{\#}}}$ and (\ref{The-2-1-Eq-1}), (\ref{The-2-1-Eq-2})
 and (\ref{The-2-1-Eq-3})   are established.

Conversely, let$\widehat{G}=G+\varepsilon R$ satisfy
 (\ref{The-2-1-Eq-1}), (\ref{The-2-1-Eq-2}),  (\ref{The-2-1-Eq-3})
  and $G=A^{\mbox{\tiny\textcircled{\#}}}$.
By applying Definition \ref{Def-2-1},
it is easy to check that $\widehat{G}$ is the DCGI of $\M$.
So,  the DCGI of ${\M}$ exists.
\end{proof}

According to Theorem \ref{The-2-1},
we can see that the existence of the core inverse of  $A$
is only a necessary condition for the dual core generalized invertibility of  ${\M}$,
that is to say,
even though the real part of a dual matrix is core invertible,
it may be also a dual matrix without DCGI.

\begin{example}
Let
\begin{align}
\label{gongshi1}
{\M}=A+\varepsilon B=
\left[\begin{matrix}
    a&         0&     0
\\  0&         b&     0
\\  0&         0&     0
\end{matrix}\right]
+\varepsilon
\left[\begin{matrix}
    b_{11}&    b_{12}&     b_{13}
\\  b_{21}&    b_{22}&     b_{23}
\\  b_{31}&    b_{32}&     b_{33}
\end{matrix}\right],
\end{align}
where $a, b$ and $b_{33}$ are not 0, $b_ {ij} \left( i=1, 2, 3,j=1, 2 \right) $,
$b_{13}$ and $b_{23}$  are arbitrary real numbers.
It is obvious that the real part $A$ is core invertible and
\begin{align}
\label{gongshi2}
A^{\mbox{\tiny\textcircled{\#}}}=
\left[\begin{matrix}
    \frac1a&         0&     0
\\  0&         \frac1b&     0
\\  0&               0&     0
\end{matrix}\right].
\end{align}

Let $\widehat{G}=A^{\mbox{\tiny\textcircled{\#}}}+\varepsilon R$, where
\begin{align}
\label{w1}
R=
 \left[\begin{matrix}
    r_{1}&    r_{2}&     r_{3}
\\  r_{4}&    r_{5}&     r_{6}
\\  r_{7}&    r_{8}&     r_{9}
\end{matrix}\right].
\end{align}

According to Theorem $\ref{The-2-1}$, if $\widehat{G}$ is the DCGI of   ${\M}$,
then $R$ is a suitable matrix of order $n$ and satisfies $(\ref{The-2-1-Eq-1})$,
$(\ref{The-2-1-Eq-2})$ and $(\ref{The-2-1-Eq-3})$.
The equation $(\ref{The-2-1-Eq-1})$ requires:
\begin{align}
\label{gao23}
\Delta
=
\underbrace{{\underbrace{BA^{\mbox{\tiny\textcircled{\#}}}A}_{S_{1}}
+
\underbrace{ARA}_{S_{2}}}
+
\underbrace{AA^{\mbox{\tiny\textcircled{\#}}}B}_{S_{3}}}_{S_{123}}-B=0.
\end{align}

Now we will prove that for any ${\M}$ constructed by $A$ and $B$ in $(\ref{gongshi1})$,
the equation $(\ref{gao23})$ does not satisfy any $3$-order matrix $R$.
Therefore, we need to prove  $S_{123}-B\neq 0$.
 As shown below,
from the above $(\ref{gongshi1})$, $(\ref{gongshi2})$ and $(\ref{w1})$, we have
\begin{align}
\nonumber
S_{1}&=BA^{\mbox{\tiny\textcircled{\#}}}A=
\left[\begin{matrix}
    b_{11}&    b_{12}&     0
\\  b_{21}&    b_{22}&     0
\\  b_{31}&    b_{32}&     0
\end{matrix}\right],
\\
\nonumber
S_{2}&=AA^{\mbox{\tiny\textcircled{\#}}}B=
\left[\begin{matrix}
    b_{11}&    b_{12}&     b_{13}
\\  b_{21}&    b_{22}&     b_{23}
\\       0&         0&       0
\end{matrix}\right],
\\
\nonumber
S_{3}&=ARA=
\left[\begin{matrix}
     a^{2}r_{1}&                             abr_{2}&         0
\\      abr_{4}&                          b^{2}r_{5}&         0
\\            0&                                   0&         0
\end{matrix}\right],
\\
\nonumber
S_{123}&=S_{1}+S_{2}+S_{3}=
\left[\begin{matrix}
    a^{2}r_{1}+2b_{11}&                        2b_{12}+abr_{2}&       b_{13}
\\     2b_{21}+abr_{4}&                       2b_{22}+b^2r_{5}&       b_{23}
\\              b_{31}&                                 b_{32}&         0
\end{matrix}\right],
\\
\nonumber
\Delta &=S_{123}-B=
\left[\begin{matrix}
     a^{2}r_{1}+b_{11}&            b_{12}+abr_{2}&                 0
\\   b_{21}+abr_{4}&               b_{22}+b^2r_{5}&                0
\\    0&                                 0&                       -b_{33}
\end{matrix}\right].
\end{align}
If $b_ {33}$ is not $0$, so $\Delta$ is not $0$ {no matter what the matrix $R$ is}.
Therefore, the core inverse condition $(\ref{The-2-1-Eq-1})$ is not satisfied, and ${\M}$ of the set has no DCGI.
\end{example}

In the following Theorem \ref{dingli3.2}, we consider the uniqueness of DCGI.
\begin{theorem}
\label{dingli3.2}
The DCGI  of any dual matrix is unique if it exists.
\end{theorem}

\begin{proof}
 Let ${\M}=A+\varepsilon B\in  {\mathbb{D}}_{n, n}$,
 $\rk \left(A\right)= r$,
and ${\M}^{\mbox{\tiny\textcircled{\#}}}=A^{\mbox{\tiny\textcircled{\#}}}+\varepsilon R$.
Suppose that $\widehat{T}$ is any DCGI of ${\M}$,
  from Theorem \ref{The-2-1} and uniqueness of the core inverse of real matrix,
we can denote $\widehat{T}$ as
$\widehat{T}=A^{\mbox{\tiny\textcircled{\#}}}+\varepsilon \widetilde{R}$. Furthermore,   write
$$X=R-\widetilde{R}.$$

Next, we prove
  $X=0$.

From  (\ref{The-2-1-Eq-1}),  it can be seen that
\begin{align}
\label{gao7}
\left\{
\begin{array}{l}
 B=BA^{\mbox{\tiny\textcircled{\#}}}A+ARA+AA^{\mbox{\tiny\textcircled{\#}}}B,
 \\
 B=BA^{\mbox{\tiny\textcircled{\#}}}A+A\widetilde{R}A+AA^{\mbox{\tiny\textcircled{\#}}}B.
\end{array}\right.
\end{align}
Through the first equation minus the second equation in (\ref{gao7}),
we get
\begin{align}
\label{gao8}
 0=A\left(R-\widetilde{R}\right)A=AXA.
\end{align}
From     (\ref{The-2-1-Eq-2}), we have
\begin{align}
\label{gao9}
\left\{
\begin{array}{l}
 R=AA^{\mbox{\tiny\textcircled{\#}}}R
 +ARA^{\mbox{\tiny\textcircled{\#}}}
 +B\left(A^{\mbox{\tiny\textcircled{\#}}}\right)^2,
  \\
  \widetilde{R}
  =AA^{\mbox{\tiny\textcircled{\#}}}\widetilde{R}
  +A\widetilde{R}A^{\mbox{\tiny\textcircled{\#}}}
  +B\left(A^{\mbox{\tiny\textcircled{\#}}}\right)^2.
\end{array}\right.
\end{align}
 Through the first equation minus the second equation in (\ref{gao9}),
we have
\begin{align}
\label{gao10}
 X=R-\widetilde{R}
 =AA^{\mbox{\tiny\textcircled{\#}}}X+AXA^{\mbox{\tiny\textcircled{\#}}}.
\end{align}
Similarly, from the (\ref{The-2-1-Eq-3}),
it can be seen that $R$ and $\widetilde{R}$   satisfy
\begin{align}
\label{gao11}
\left\{
\begin{array}{l}
 \left(AR+BA^{\mbox{\tiny\textcircled{\#}}}\right)^\mathrm{ T }
 =AR+BA^{\mbox{\tiny\textcircled{\#}}},
 \\
  \left(A\widetilde{R}+BA^{\mbox{\tiny\textcircled{\#}}}\right)^\mathrm{ T }
  =A\widetilde{R}+BA^{\mbox{\tiny\textcircled{\#}}}.
\end{array}\right.
\end{align}
 Through the first equation minus the second equation in (\ref{gao11}),
we have
 $\left(A\left(R-\widetilde{R}\right)\right)^\mathrm{ T }
 =A\left(R-\widetilde{R}\right)$,
that is,
\begin{align}
\label{gao12}
 \left(AX\right)^\mathrm{ T }=AX.
\end{align}

 By post-multiplying both sides of   (\ref{gao8})
  by $A^{\mbox{\tiny\textcircled{\#}}}$,
 and
 by applying   (\ref{gao12})    we get
\begin{align}
\nonumber
0&=AXA=AXAA^{\mbox{\tiny\textcircled{\#}}}
=\left(AX\right)^\mathrm{ T }\left(AA^{\mbox{\tiny\textcircled{\#}}}\right)^\mathrm{ T }
=X^\mathrm{ T }A^\mathrm{ T }\left(A^{\mbox{\tiny\textcircled{\#}}}\right)^\mathrm{ T }A^\mathrm{ T }
\\
\nonumber
&=X^\mathrm{ T }\left(AA^{\mbox{\tiny\textcircled{\#}}}A\right)^\mathrm{ T }
=X^\mathrm{ T }A^\mathrm{ T }=(AX)^\mathrm{ T }=AX,
\end{align}
that is,
 $AX=0$.
Thus,
the equation (\ref{gao10}) is simplified to
\begin{align}
\label{core-core}
X=R-\widetilde{R}=AA^{\mbox{\tiny\textcircled{\#}}}X.
\end{align}

Let   the decomposition of   $A$ be as  in $(\ref{core-1})$.
Write
\begin{align}
\label{c-c-1}
X=
U\left[\begin{matrix}
      X_1&     X_2
\\    X_3&     X_4
\end{matrix}\right]U^\mathrm{ T },
\end{align}
where $X_1\in \mathbb{R}_{r, r}$.
Put   (\ref{core-1}),  (\ref{core-6}) and (\ref{c-c-1})  into (\ref{core-core}),
then
we get
\begin{align}
\nonumber
AA^{\mbox{\tiny\textcircled{\#}}}X
&=U\left[\begin{matrix}
     \Sigma K&     \Sigma L
\\          0&         0
\end{matrix}\right]U^\mathrm{ T }
U\left[\begin{matrix}
     (\Sigma K)^{-1}&     0
\\                 0&     0
\end{matrix}\right]U^\mathrm{ T }
U\left[\begin{matrix}
      X_1&     X_2
\\    X_3&     X_4
\end{matrix}\right]U^\mathrm{ T }
\nonumber
\\
&=U\left[\begin{matrix}
     I_r&     0
\\     0&     0
\end{matrix}\right]U^\mathrm{ T }
U\left[\begin{matrix}
      X_1&     X_2
\\    X_3&     X_4
\end{matrix}\right]U^\mathrm{ T }
\nonumber
\\
\nonumber
&=U\left[\begin{matrix}
      X_1&     X_2
\\      0&      0
\end{matrix}\right]U^\mathrm{ T }=
U\left[\begin{matrix}
      X_1&     X_2
\\    X_3&     X_4
\end{matrix}\right]U^\mathrm{ T }=X.
\end{align}
Therefore, $X_3=X_4=0$, that is,
\begin{align}
\label{X-jianhua}
X=U\left[\begin{matrix}
      X_1&     X_2
\\      0&     0
\end{matrix}\right]U^\mathrm{ T }.
\end{align}
 And to make $X=0$ hold, we only need to prove $X_1=0$ and $X_2=0$.

By substituting (\ref{core-1}) and (\ref{X-jianhua}) into (\ref{gao8}),
we get
\begin{align*}
AXA
&=U\left[\begin{matrix}
          \Sigma K&     \Sigma L
\\               0&         0
\end{matrix}\right]U^\mathrm{ T }
U\left[\begin{matrix}
      X_1&     X_2
\\      0&     0
\end{matrix}\right]U^\mathrm{ T }
U\left[\begin{matrix}
          \Sigma K&     \Sigma L
\\               0&         0
\end{matrix}\right]U^\mathrm{ T }
\nonumber
\\
&=U\left[\begin{matrix}
              \Sigma KX_1&         \Sigma KX_2
\\                      0&                   0
\end{matrix}\right]U^\mathrm{ T }
U\left[\begin{matrix}
          \Sigma K&      \Sigma L
\\               0&         0
\end{matrix}\right]U^\mathrm{ T }
\nonumber
\\
&=U\left[\begin{matrix}
    \Sigma KX_1\Sigma K&    \Sigma KX_1\Sigma L
\\                    0&                      0
\end{matrix}\right]U^\mathrm{ T }=0.
\end{align*}
Therefore,
\begin{subnumcases}{}
 \label{1-a}
 \Sigma KX_1\Sigma K=0,
 \\
 \label{1-b}
 \Sigma KX_1\Sigma L=0.
 \end{subnumcases}
Since $A$ is core invertible,
from Theorem \ref{The-2-1}, we see that $K$ is nonsingular.
 Thus,
  from (\ref{X-jianhua}) and  (\ref{1-a}),
  we   get $X_{1}=0$.
    So \begin{align}
\label{X-jianhua2}
X=U\left[\begin{matrix}
        0&     X_2
\\      0&     0
\end{matrix}\right]U^\mathrm{ T }.
\end{align}

Similarly,
by substituting (\ref{core-1}) and (\ref{X-jianhua2}) into (\ref{gao12}),
we get
\begin{align*}
\left(U\left[\begin{matrix}
          \Sigma K&     \Sigma L
\\               0&         0
\end{matrix}\right]
\left[\begin{matrix}
      0&     X_2
\\    0&     0
\end{matrix}\right]U^\mathrm{ T }\right)^\mathrm{ T }
=
U\left[\begin{matrix}
          \Sigma K&     \Sigma L
\\               0&         0
\end{matrix}\right]
\left[\begin{matrix}
      0&     X_2
\\    0&     0
\end{matrix}\right]U^\mathrm{ T },
\end{align*}
that is,
\begin{align*}
\left(U\left[\begin{matrix}
                                0&           \Sigma KX_2
\\                              0&                0
\end{matrix}\right]U^\mathrm{ T }\right)^\mathrm{ T }
=
U\left[\begin{matrix}
                                 0&          \Sigma KX_2
\\                              0&                0
\end{matrix}\right]U^\mathrm{ T }.
\end{align*}
Continue to simplify the above equation and get
\begin{align*}
U\left[\begin{matrix}
                                              0&                0
\\        \left(\Sigma KX_2\right)^\mathrm{ T }&                0
\end{matrix}\right]U^\mathrm{ T }
=
U\left[\begin{matrix}
                                0&            \Sigma KX_2
\\                              0&                0
\end{matrix}\right]U^\mathrm{ T }.
\end{align*}
 Thus,
$\Sigma KX_2=0$.
Considering that both $\Sigma$ and $K$ are nonsingular matrices,
we get $X_2=0.$

To sum up, we obtain $X_1=0$, $X_2=0$, $X_3=0$ and $X_4=0$.
 From (\ref{c-c-1}), we get  $X=0$, which can also be understood as
$$R=\widetilde{R}.$$
Therefore,
$R$ satisfying
 (\ref{The-2-1-Eq-1}),
  (\ref{The-2-1-Eq-2}) and (\ref{The-2-1-Eq-3})
is unique,
that is,
if the DCGI of   ${\M}$ exists,
then the inverse is unique.
\end{proof}

\subsection{Characterizations and  properties  of DCGI}

\begin{theorem}
\label{core-20}
Let ${\M}=A+\varepsilon B\in  {\mathbb{D}}_{n, n}$,
 then its DCGI exists if and only if its dual index is one.
\end{theorem}

\begin{proof}
Suppose that $A$ is core invertible.
Let the decomposition of $A$ be as in $(\ref{core-1})$,
and the form of $B$ be as in (\ref{core-7}).

``$\Longrightarrow$''
\
Assuming that the dual core inverse
${\M}^{\mbox{\tiny\textcircled{\#}}}=G+\varepsilon R$ of the dual matrix ${\M}=A+\varepsilon B$ exists,
it can be seen
from Theorem \ref{The-2-1}
 that
 the real part matrix $A$ is core invertible
and
$G=A^{\mbox{\tiny\textcircled{\#}}}$,
so the index of $A$ is one.

 Since the DCGI exists, then the equation (\ref{The-2-1-Eq-1}) holds.
  Put (\ref{core-1}), (\ref{core-6}) and (\ref{core-7})
  into
  $BGA + ARA + AGB = B$,
  we get
\begin{align}
\nonumber
&
U\left[\begin{matrix}
     B_1&         B_2
\\   B_3&         B_4
\end{matrix}\right]U^\mathrm{ T }
\\
\nonumber
&\quad=
U\left[\begin{matrix}
      B_1+\left(\Sigma KR_1+ \Sigma LR_3\right) \Sigma K+B_1&         B_1K^{-1}L+\left(\Sigma KR_1+ \Sigma LR_3\right) \Sigma L+ B_2
\\                                            B_3&                               B_3K^{-1}L
\end{matrix}\right]U^\mathrm{ T }.
\end{align}
So we have $B_4=B_3K^{-1}L$.

To sum up, from Theorem \ref{The-2-40},
when the index of $A$ is one and $B_4=B_3K^{-1}L$,
the dual index of ${\M}$ is one.

\

``$\Longleftarrow$''
\
Let the dual index of ${\M}$ be one.
According to Theorem \ref{The-2-40},
 the index of $A$ is one and $B_4=B_3K^{-1}L$.
It  follows from (\ref{core-7}) that
\begin{align}
\label{core-88}
B=U\left[\begin{matrix}
     B_1&         B_2
\\   B_3&         B_3K^{-1}L
\end{matrix}\right]U^\mathrm{ T }.
\end{align}

From the above equation (\ref{core-88}) and equation (\ref{core-1}), we get
\begin{align*}
{\M}
&=A+\varepsilon B
 =U\left[\begin{matrix}
     \Sigma K&         \Sigma L
\\          0&             0
\end{matrix}\right]U^\mathrm{ T }+
\varepsilon U\left[\begin{matrix}
     B_1&         B_2
\\   B_3&         B_3K^{-1}L
\end{matrix}\right]U^\mathrm{ T }.
\end{align*}

Denote
\begin{align}
\nonumber
\widehat{G}
=
G
&
+\varepsilon R
=
 U\left[\begin{matrix}
     \left(\Sigma K\right)^{-1}&      0
\\                0&                  0
\end{matrix}\right]U^\mathrm{ T }
\\
\label{G-gongshi}
&
+
\varepsilon U\left[\begin{matrix}
     -K^{-1}LB_3\left(\Sigma K\right)^{-2}-\left(\Sigma K\right)^{-1}B_1\left(\Sigma K\right)^{-1}
     &         K^{-1}\Sigma^{-2}\left(B_3K^{-1}\right)^\mathrm{ T }
\\                                             B_3(\Sigma K)^{-2}&           0
\end{matrix}\right]U^\mathrm{ T }.
\end{align}
Then
\begin{align*}
{\M}\widehat{G}{\M}
&=AGA+\varepsilon \left(BGA+ARA+AGB\right)
\\
&=U\left[\begin{matrix}
                \Sigma K&      \Sigma L
\\                     0&         0
\end{matrix}\right]U^\mathrm{ T }
\\
&
\quad
+\varepsilon
\left(U\left[\begin{matrix}
                   B_1&      B_1\left(\Sigma K\right)^{-1}\Sigma L
\\                 B_3&      B_3\left(\Sigma K\right)^{-1}\Sigma L
\end{matrix}\right]U^\mathrm{ T }+
U\left[\begin{matrix}
                    -B_1&      -B_1\left(\Sigma K\right)^{-1}\Sigma L
\\                     0&                  0
\end{matrix}\right]U^\mathrm{ T }+
U\left[\begin{matrix}
                       B_1&      B_2
\\                       0&       0
\end{matrix}\right]U^\mathrm{ T }\right)
\\
&=U\left[\begin{matrix}
                \Sigma K&      \Sigma L
\\                     0&         0
\end{matrix}\right]U^\mathrm{ T }+\varepsilon
U\left[\begin{matrix}
                   B_1&      B_2
\\                 B_3&      B_3K^{-1}L
\end{matrix}\right]U^\mathrm{ T }
\\
&=A+\varepsilon B={\M},
\\
{\M}\widehat{G}^2
&=AG^2+\varepsilon \left(AGR+ARG+BG^2\right)
\\
&=U\left[\begin{matrix}
                \left(\Sigma K\right)^{-1}&  0
\\                     0&         0
\end{matrix}\right]U^\mathrm{ T }+\varepsilon
U\left(\left[\begin{matrix}
                   -K^{-1}LB_3\left(\Sigma K\right)^{-2}-\left(\Sigma K\right)^{-1}B_1\left(\Sigma K\right)^{-1}
                   &      K^{-1}\Sigma^{-2}(B_3K^{-1})^\mathrm{ T }
\\      0&   0
\end{matrix}\right]\right)U^\mathrm{ T }
\\
&+\varepsilon U\left(\left[\begin{matrix}
                    -B_1(\Sigma K)^{-2}&                  0
\\                                    0&                  0
\end{matrix}\right]+
\left[\begin{matrix}
             B_1(\Sigma K)^{-2}&       0
\\           B_3(\Sigma K)^{-2}&       0
\end{matrix}\right]\right)U^\mathrm{ T }
\\
&=U\left[\begin{matrix}
                (\Sigma K)^{-1}&         0
\\                            0&         0
\end{matrix}\right]U^\mathrm{ T }+\varepsilon
U\left[\begin{matrix}
              -K^{-1}LB_3(\Sigma K)^{-2}-(\Sigma K)^{-1}B_1(\Sigma K)^{-1}
              &      K^{-1}\Sigma^{-2}\left(B_3K^{-1}\right)^\mathrm{ T }
\\    B_3\left(\Sigma K\right)^{-2}&                       0
\end{matrix}\right]U^\mathrm{ T }
\\
&=G+\varepsilon R=\widehat{G}
\end{align*}
and
\begin{align*}
\left({\M}\widehat{G}\right)^\mathrm{ T }
&=\left(AG+\varepsilon (AR+BG)\right)^\mathrm{ T }
\\
&=\left(U\left[\begin{matrix}
                I_r&      0
\\                0&       0
\end{matrix}\right]U^\mathrm{ T }\right)^\mathrm{ T }+\varepsilon
\left(U\left[\begin{matrix}
                        0&                   \Sigma^{-1}\left(B_3K^{-1}\right)^\mathrm{ T }
\\     B_3\left(\Sigma K\right)^{-1}&                         0
\end{matrix}\right]U^\mathrm{ T }\right)^\mathrm{ T }
\\
&=AG+\varepsilon \left(AR+BG\right)={\M}\widehat{G}.
\end{align*}
Therefore, $\widehat{G}$ is the DCGI of ${\M}$,
 that is, $\widehat{G}={\M}^{\mbox{\tiny\textcircled{\#}}}$ from Theorem \ref{The-2-1}.
\end{proof}

\begin{theorem}
\label{core-dengjia}
Let ${\M}=A+\varepsilon B\in  {\mathbb{D}}_{n, n}$,
then the DCGI ${\M}^{\mbox{\tiny\textcircled{\#}}}$ of ${\M}$ exists
if and only if
${\rm Ind}\left(A\right)=1$
and
$\left(I-AA^{+}\right)B\left(I-AA^{\#}\right)=0$.
\end{theorem}

\begin{proof}

Suppose that $A$ is core invertible.
Let   the decomposition of   $A$ be as  in $(\ref{core-1})$,
and the form of $B$ be as in (\ref{core-7}).

``$\Longrightarrow$''
\
Let the DCGI ${\M}^{\mbox{\tiny\textcircled{\#}}}$ of ${\M}$ exist,
then we have
 the index of $A$ is one,
and $B_4=B_3K^{-1}L$ from Theorem \ref{core-20}.
Thus, we get equation (\ref{core-88}).
 Put equations (\ref{core-1}), (\ref{core-3}), (\ref{core-4}) and (\ref{core-88})
 into $(I-AA^{+})B(I-AA^{\#})$ to get
\begin{align}
\left(I-AA^{+}\right)B\left(I-AA^{\#}\right)
&=U\left[\begin{matrix}
     0&         0
\\   0&         I_{n-r}
\end{matrix}\right]
\left[\begin{matrix}
     B_1&         B_2
\\   B_3&         B_3K^{-1}L
\end{matrix}\right]
\left[\begin{matrix}
       0&       -K^{-1}L
\\     0&       I_{n-r}
\end{matrix}\right]U^\mathrm{ T }
\\
\nonumber
&=U\left[\begin{matrix}
       0&         0
\\   B_3&         B_3K^{-1}L
\end{matrix}\right]
\left[\begin{matrix}
       0&       -K^{-1}L
\\     0&       I_{n-r}
\end{matrix}\right]U^\mathrm{ T }
\\
\nonumber
&=U\left[\begin{matrix}
       0&       0
\\     0&       -B_3K^{-1}L+B_3K^{-1}L
\end{matrix}\right]U^\mathrm{ T }=0
\end{align}

``$\Longleftarrow$''
\
Let $\left(I-AA^{+}\right)B\left(I-AA^{\#}\right)=0$.
Since $A^{\#}$ exists,
 the index of $A$ is one.
Put equations (\ref{core-1}), (\ref{core-3}), (\ref{core-4}) and (\ref{core-7})
into $\left(I-AA^{+}\right)B\left(I-AA^{\#}\right)$ to get
\begin{align}
\left(I-AA^{+}\right)B\left(I-AA^{\#}\right)
&=U\left[\begin{matrix}
     0&         0
\\   0&         I_{n-r}
\end{matrix}\right]
\left[\begin{matrix}
     B_1&         B_2
\\   B_3&         B_4
\end{matrix}\right]
\left[\begin{matrix}
       0&       -K^{-1}L
\\     0&       I_{n-r}
\end{matrix}\right]U^\mathrm{ T }
\\
\nonumber
&=U\left[\begin{matrix}
       0&         0
\\   B_3&         B_4
\end{matrix}\right]
\left[\begin{matrix}
       0&       -K^{-1}L
\\     0&       I_{n-r}
\end{matrix}\right]U^\mathrm{ T }
\\
\nonumber
&=U\left[\begin{matrix}
       0&       0
\\     0&       -B_3K^{-1}L+B_4
\end{matrix}\right]U^\mathrm{ T }.
\end{align}
It follows from $\left(I-AA^{+}\right)B\left(I-AA^{\#}\right)=0$
  that $B_4=B_3K^{-1}L$. Then there is equation (\ref{G-gongshi})
 from Theorem \ref{core-20}.
 According to Definition \ref{Def-2-1}, DCGI exists.
\end{proof}

\begin{theorem}
\label{core-dengjia-2}
Let ${\M}=A+\varepsilon B\in  {\mathbb{D}}_{n, n}$,
then the DCGI ${\M}^{\mbox{\tiny\textcircled{\#}}}$ of ${\M}$ exists
if and only if
${\rm Ind}\left(A\right)=1$
and
$\left(I-AA^{\mbox{\tiny\textcircled{\#}}}\right)B\left(I-A^{\mbox{\tiny\textcircled{\#}}}A\right)=0$.
\end{theorem}

\begin{proof}

In real field, it is known that
$A^{\#}A=AA^{\#}=A^{\mbox{\tiny\textcircled{\#}}}A$,
 and
$AA^{+}=AA^{\mbox{\tiny\textcircled{\#}}}$.
According to Theorem \ref{core-dengjia},
the DCGI ${\M}^{\mbox{\tiny\textcircled{\#}}}$ of ${\M}$ exists
if and only if
${\rm Ind}\left(A\right)=1$
and
$\left(I-AA^{\mbox{\tiny\textcircled{\#}}}\right) B \left(I-A^{\mbox{\tiny\textcircled{\#}}}A\right)=0$.
\end{proof}

  Next, we further discuss  characterizations of the existence of DCGI.

\begin{theorem}
\label{TLGUANXI}
Let ${\M}=A+\varepsilon B\in{\mathbb{D}}_{n, n}$, then the following conditions are equivalent:
\begin{description}
  \item[(1)]
  The  DCGI of ${\M}$ exists;

  \item[(2)]
  The index of $A$ is equal to 1, and
$BA^{\#}A+AXA+AA^{\#}B=B$ is consistent;

  \item[(3)]
  The index of $A$ is equal to 1, and
$BA^+A+AXA+AA^+B=B$ is consistent;

  \item[(4)]
  The index of $A$ is equal to 1, and
$BA^{\mbox{\tiny\textcircled{\#}}}A+AXA+AA^{\mbox{\tiny\textcircled{\#}}}B=B$ is consistent.
\end{description}
\end{theorem}

\begin{proof}

From Theorem \ref{zhibiaoqunni-xin}, Theorem \ref{zhibiaoMPni} and Theorem \ref{core-20},
we know that the existence of DCGI is equivalent to the existence of DGGI;
the existence of DGGI is equivalent to
the existence of DMPGI with ${\rm Ind}(A)=1$.
Therefore, condition (1) indicates that
DCGI exists
or
DGGI exists
or
DMPGI exists with ${\rm Ind}(A)=1$.

``$(1)\Longrightarrow(2)$''
\
Let DCGI exist, then ${\M}^{\#}=A^{\#}+\varepsilon P$ of ${\M}=A+\varepsilon B$ exist,
so the index of $A$ is one and ${\M}{\M}^{\#}{\M}={\M}$,
that is,
$BA^{\#}A+APA+AA^{\#}B=B$ .
Therefore, $BA^{\#}A+AXA+AA^{\#}B=B$ is consistent.

``$(1)\Longleftarrow(2)$''
\
Let the index of $A$ be one and $BA^{\#}A+AXA+AA^{\#}B=B$ be consistent.
 By applying   (\ref{core-1}),
 (\ref{core-2}), (\ref{core-4}) and (\ref{core-7})
 to $BA^{\#}A+AXA+AA^{\#}B=B$, we get
 \begin{align}
 \label{mafan}
 &U\left[\begin{matrix}
     B_1&         B_2
\\   B_3&         B_4
\end{matrix}\right]
\left[\begin{matrix}
      K^{-1}\Sigma^{-1}&     K^{-1}\Sigma^{-1}K^{-1}L
\\                    0&         0
\end{matrix}\right]
\left[\begin{matrix}
     \Sigma K&         \Sigma L
\\        0&         0
\end{matrix}\right]U^\mathrm{ T }
\\
\nonumber
&+U\left[\begin{matrix}
     \Sigma K&         \Sigma L
\\        0&         0
\end{matrix}\right]
\left[\begin{matrix}
     X_1&         X_2
\\   X_3&         X_4
\end{matrix}\right]
\left[\begin{matrix}
     \Sigma K&         \Sigma L
\\        0&         0
\end{matrix}\right]U^\mathrm{ T }
\\
\nonumber
&+U\left[\begin{matrix}
     \Sigma K&         \Sigma L
\\        0&         0
\end{matrix}\right]
\left[\begin{matrix}
      K^{-1}\Sigma^{-1}&     K^{-1}\Sigma^{-1}K^{-1}L
\\                    0&         0
\end{matrix}\right]
\left[\begin{matrix}
     B_1&         B_2
\\   B_3&         B_4
\end{matrix}\right]U^\mathrm{ T }=
U\left[\begin{matrix}
     B_1&         B_2
\\   B_3&         B_4
\end{matrix}\right]U^\mathrm{ T },
 \end{align}
 where
$X= U\left[\begin{matrix}
     X_1&         X_2
\\   X_3&         X_4
\end{matrix}\right]U^\mathrm{ T }$,
 $B_1$ is an $r$-square matrix, and $r=\rk(A)$.

By simplifying equation (\ref{mafan}), we get
\begin{align}U\left[\begin{matrix}
     B_1&         B_2
\\   B_3&         B_4
\end{matrix}\right]U^\mathrm{ T }
\nonumber
&=
U\left[\begin{matrix}
     B_1+\Sigma KX_{1}\Sigma K+\Sigma LX_{3}\Sigma K+B_{1}+K^{-1}LB_{3}
\\   B_3
\end{matrix}\right.
\\
&
\qquad\qquad\qquad
 \left.\begin{matrix}
           B_1K^{-1}L+\Sigma KX_{1}\Sigma L+\Sigma LX_{3}\Sigma L+B_{2}+K^{-1}LB_{4}
\\         B_3K^{-1}L
\end{matrix}\right]U^\mathrm{ T }
\nonumber,
\end{align}
Therefore, $B_4=B_3K^{-1}L$.

 To sum up, if ${\rm Ind}\left(A\right)=1$ and $B_4=B_3K^{-1}L$,
 then the dual index of ${\M}$ is one from Theorem \ref{The-2-40}.
  And because of  Theorem \ref{core-20}, we know the  DCGI of ${\M}$ exists.

Similarly, (1) and (3) are equivalent;
 (1) and (4) are equivalent.
\end{proof}

\subsection{Compact formula for DCGI}

The following is a compact formula for DCGI.

\begin{theorem}
\label{c-jinhua}
Let ${\M}=A+\varepsilon B\in{\mathbb{D}}_{n, n}$ and the DCGI of ${\M}$ exist, then
\begin{align}
\label{jinjinhua}
{\M}^{\mbox{\tiny\textcircled{\#}}}
&={\M}^{\#}{\M}{\M}^{+}
\\
\nonumber
&=
A^{\mbox{\tiny\textcircled{\#}}}+\varepsilon\left(-A^{\mbox{\tiny\textcircled{\#}}}BA^{+}
+A^{\#}BA^{+}-A^{\#}BA^{\mbox{\tiny\textcircled{\#}}}
+A^{\mbox{\tiny\textcircled{\#}}}(BA^{+})^\mathrm{ T }(I_n-AA^{+})\right.
\\
\label{changjinhua}
&
\qquad\qquad\quad
+\left.(I_n-AA^{\#})BA^{\#}A^{\mbox{\tiny\textcircled{\#}}}\right).
\end{align}
\end{theorem}

\begin{proof}
According to  Theorem \ref{zhibiaoqunni-xin}, Theorem \ref{zhibiaoMPni} and Theorem \ref{core-20},
if ${\M}^{\mbox{\tiny\textcircled{\#}}}$ exists, then ${\M}^{\#}$ and ${\M}^{+}$ exist.
 Write
 $\widehat{X}={\M}^{\#}{\M}{\M}^{+}$.
 It is easy to check that
\begin{align}
{\M}\widehat{X}{\M}&=  {\M}{\M}^{\#}{\M} {\M}^{+}{\M}={\M}{\M}^{+}{\M}={\M},
\\
{\M}\widehat{X}^2&= {\M}{\M}^{\#}{\M} {\M}^{+} {\M}^{\#}{\M} {\M}^{+}=
 {\M}{\M}^{+}{\M} {\M}^{\#}{\M}^{+}
= {\M}{\M}^{\#} {\M}^{+}={\M}^{\#}{\M}{\M}^{+}=\widehat{X}.
\end{align}
Since
$\left({\M}\widehat{X}\right)^\mathrm{ T }=\left( {\M}{\M}^{\#}{\M} {\M}^{+}\right)^\mathrm{ T }=
\left({\M}{\M}^{+}\right)^\mathrm{ T }={\M}{\M}^{+}$
and
${\M}\widehat{X}= {\M}{\M}^{\#}{\M} {\M}^{+}={\M}{\M}^{+}$,
we get
\begin{align}
\left({\M}\widehat{X}\right)^\mathrm{ T }={\M}\widehat{X}.
\end{align}

To sum up, we get
${\M}^{\mbox{\tiny\textcircled{\#}}}=\widehat{X}={\M}^{\#}{\M}{\M}^{+}$.

Substituting  (\ref{The-DMPGI})
and  (\ref{qunni-gongshi}) into (\ref{jinjinhua})
gives (\ref{changjinhua}).
\end{proof}

In addition,
substituting   (\ref{core-1}), (\ref{core-2}),
  (\ref{core-3}),  (\ref{core-6}), (\ref{core-4})
  and (\ref{core-88})   into (\ref{changjinhua}),
we get the following Theorem \ref{20220528-1}.

\begin{theorem}
\label{20220528-1}
Let the DCGI of ${\M}=A+\varepsilon B\in  {\mathbb{D}}_{n, n}$ exist;
$A$ and  $B$ be as forms in $(\ref{core-1})$ and $(\ref{core-7})$, respectively.
Then
\begin{align}
\nonumber
{\M}^{\mbox{\tiny\textcircled{\#}}}
=
&
U\left[\begin{matrix}
     \left(\Sigma K\right)^{-1}&      0
\\                0&       0
\end{matrix}\right]U^\mathrm{ T }
\\
\label{xiaofenjie}
&
\quad +
\varepsilon U\left[\begin{matrix}
     -K^{-1}LB_3\left(\Sigma K\right)^{-2}-\left(\Sigma K\right)^{-1}B_1\left(\Sigma K\right)^{-1}
     &         K^{-1}\Sigma^{-2}\left(B_3K^{-1}\right)^\mathrm{ T }
\\                                             B_3\left(\Sigma K\right)^{-2}&                              0
\end{matrix}\right]U^\mathrm{ T }.
\end{align}
\end{theorem}

\subsection{Relationships among some dual generalized inverses}

There is a very interesting inverse,
the Moore-Penrose dual generalized inverse (MPDGI):
 \begin{align}
 \label{P}
 {\M}^{\PP}=A^{+}-\varepsilon A^{+}BA^{+}
 \end{align}
The inverse is useful for solving different kinematic problems \cite{d1g}.
Obviously,
when the real part of a dual matrix is a nonsingular matrix,
its MPDGI is equal to its DMPGI.
Similarly,
we consider
  laws and properties
of DCGI in the form of
${\M}^{\mbox{\tiny\textcircled{\#}}}
=A^{\mbox{\tiny\textcircled{\#}}}-\varepsilon {A^{\mbox{\tiny\textcircled{\#}}}BA^{\mbox{\tiny\textcircled{\#}}}}$ and
DGGI in the form of
 ${\M}^{\#}=A^{\#}-\varepsilon A^{\#}BA^{\#}$,
and
 relationships among those dual generalized inverses.

\begin{theorem}
\label{suibian}
 Let the DCGI ${\M}^{\mbox{\tiny\textcircled{\#}}}$ of  ${\M}=A+\varepsilon B\in{\mathbb{D}}_{n, n}$   exist,
 where $\rk\left(A\right)=r$.
Then
 ${\M}^{\mbox{\tiny\textcircled{\#}}}
 =A^{\mbox{\tiny\textcircled{\#}}}-\varepsilon {A^{\mbox{\tiny\textcircled{\#}}}BA^{\mbox{\tiny\textcircled{\#}}}}$,
 which is equivalent to
${\rm Ind}\left(A\right)=1$ and
\begin{align}
\label{core-jianhuan}
\left(I_n-AA^{+}\right)B=0.
\end{align}
\end{theorem}

\begin{proof}
Suppose that $A$ is core invertible.
Let   the decomposition of   $A$ be as  in $(\ref{core-1})$,
and the form of $B$ be as in (\ref{core-7}).
Then
\begin{align}
\label{zhangkai-20220524}
-{A^{\mbox{\tiny\textcircled{\#}}}BA^{\mbox{\tiny\textcircled{\#}}}}
&=
 U\left[\begin{matrix}
       -\left(\Sigma K\right)^{-1}B_{1}\left(\Sigma K\right)^{-1}&       0
\\                                       0&                              0
\end{matrix}\right]U^\mathrm{ T },
\\
\label{zhangkai-1}
 -A^{\mbox{\tiny\textcircled{\#}}}BA^{+}+A^{\#}BA^{+}
 &
 =
 U\left[\begin{matrix}
       \left(\Sigma K\right)^{-1}K^{-1}L\left(B_{3}K^\mathrm{ T }
       +B_{4}L^\mathrm{ T }\right)\Sigma^{-1}
       &       0
\\               0&               0
\end{matrix}\right]U^\mathrm{ T },
\\
\label{zhangkai-2}
-A^{\#}BA^{\mbox{\tiny\textcircled{\#}}}
  =
  U
 &
 \left[\begin{matrix}
       -\left(\Sigma K\right)^{-1}B_{1}\left(\Sigma K\right)^{-1}
       -\left(\Sigma K\right)^{-1}K^{-1}LB_{3}\left(\Sigma K\right)^{-1}&       0
\\                                                                                                                               0&       0
\end{matrix}\right]U^\mathrm{ T },
\\
\label{zhangkai-3}
 A^{\mbox{\tiny\textcircled{\#}}}\left(BA^{+}\right)^\mathrm{ T }\left(I_n-AA^{+}\right)
 &
 =
 U\left[\begin{matrix}
      0&        \left(\Sigma K\right)^{-1}\left(B_{3}K^\mathrm{ T }\Sigma^{-1}
      +B_{4}L^\mathrm{ T }\Sigma^{-1}\right)^\mathrm{ T }
\\    0&                             0
\end{matrix}\right]U^\mathrm{ T },
\\
\label{zhangkai-4}
\left(I_n-AA^{\#}\right)BA^{\#}A^{\mbox{\tiny\textcircled{\#}}}
 &
 =
 U\left[\begin{matrix}
      -K^{-1}LB_{3}\left(\Sigma K\right)^{-2}&        0
\\            B_{3}\left(\Sigma K\right)^{-2}&        0
\end{matrix}\right]U^\mathrm{ T },
\\
 \label{xiamian}
\left(I_n-AA^{+}\right)B
&
=
U\left[\begin{matrix}
      0&        0
\\    0&        I_{n-r}
\end{matrix}\right]
\left[\begin{matrix}
      B_1&        B_2
\\    B_3&        B_4
\end{matrix}\right]
U^\mathrm{ T }
=
U\left[\begin{matrix}
        0&        0
\\    B_3&        B_4
\end{matrix}\right]
U^\mathrm{ T }.
\end{align}

``$\Longrightarrow$''
\
Let the DCGI ${\M}^{\mbox{\tiny\textcircled{\#}}}$ of
${\M}
=A+\varepsilon B\in{\mathbb{D}}_{n, n}$  exist,
 and ${\M}^{\mbox{\tiny\textcircled{\#}}}
 =A^{\mbox{\tiny\textcircled{\#}}}-\varepsilon {A^{\mbox{\tiny\textcircled{\#}}}BA^{\mbox{\tiny\textcircled{\#}}}}$.
According to Theorem \ref{core-20},
we get ${\rm Ind}\left(A\right)=1$
 and $B_4=B_3K^{-1}L$.
 To make ${\M}^{\mbox{\tiny\textcircled{\#}}}
 =A^{\mbox{\tiny\textcircled{\#}}}-\varepsilon {A^{\mbox{\tiny\textcircled{\#}}}BA^{\mbox{\tiny\textcircled{\#}}}}$,
from the compact formula for DCGI   (\ref{changjinhua}),
we get
\begin{align*}
 {A^{\mbox{\tiny\textcircled{\#}}}BA^{\mbox{\tiny\textcircled{\#}}}}
 =
 -A^{\mbox{\tiny\textcircled{\#}}}BA^{+}+A^{\#}BA^{+}
 -A^{\#}BA^{\mbox{\tiny\textcircled{\#}}}
 +A^{\mbox{\tiny\textcircled{\#}}}\left(BA^{+}\right)^\mathrm{ T }\left(I_n-AA^{+}\right)
 +\left(I_n-AA^{\#}\right)BA^{\#}A^{\mbox{\tiny\textcircled{\#}}}.
\end{align*}
It follows from
 (\ref{zhangkai-20220524}), (\ref{zhangkai-1}),
  (\ref{zhangkai-2}), (\ref{zhangkai-3}) and (\ref{zhangkai-4})
  that
  $B_{3}\left(\Sigma K\right)^{-2}=0$, so $B_{3}=0$.
 Since $B_4=B_3K^{-1}L$, then $B_4=0$.
 By applying  (\ref{xiamian}), $B_{3}=0$ and  $B_4=0$,
we derive $\left(I_n-AA^{+}\right)B=0$.
Thus, ${\rm Ind}\left(A\right)=1$ and $\left(I_n-AA^{+}\right)B=0$ are proved.

``$\Longleftarrow$''
 \
Let ${\rm Ind}\left(A\right)=1$ and $\left(I_n-AA^{+}\right)B=0$,
then $\left(I_n-AA^{+}\right)B\left(I_n-AA^{\#}\right)=0$.
According to Theorem
\ref{zhibiaoMP}, the DCGI of  ${\M}$ exists.
We know that the real part of the DCGI is $A^{\mbox{\tiny\textcircled{\#}}}$,
that is,
${\M}^{\mbox{\tiny\textcircled{\#}}}=A^{\mbox{\tiny\textcircled{\#}}}+\varepsilon R.$
 Put equations (\ref{core-1}), (\ref{core-3})
 and (\ref{core-7}) into (\ref{xiamian})
 to get  $\left(I_n-AA^{+}\right)B=0$.
So $B_3=0$ and $B_4=0$. Put $B_3=0$ and $B_4=0$ into (\ref{changjinhua}) to get
\begin{align}
\nonumber
R
&=-A^{\mbox{\tiny\textcircled{\#}}}BA^{+}+A^{\#}BA^{+}-A^{\#}BA^{\mbox{\tiny\textcircled{\#}}}
+A^{\mbox{\tiny\textcircled{\#}}}(BA^{+})^\mathrm{ T }\left(I_n-AA^{+}\right)
+\left(I_n-AA^{\#}\right)BA^{\#}A^{\mbox{\tiny\textcircled{\#}}}
\\
\label{000}
&=U\left[\begin{matrix}
       -\left(\Sigma K\right)^{-1}B_{1}\left(\Sigma K\right)^{-1}&       0
\\                                       0&                              0
\end{matrix}\right]U^\mathrm{ T }=-{A^{\mbox{\tiny\textcircled{\#}}}BA^{\mbox{\tiny\textcircled{\#}}}}.
\end{align}
Therefore,
 ${\M}^{\mbox{\tiny\textcircled{\#}}}
 =A^{\mbox{\tiny\textcircled{\#}}}-\varepsilon {A^{\mbox{\tiny\textcircled{\#}}}BA^{\mbox{\tiny\textcircled{\#}}}}$.
\end{proof}

Since it is well known that $\left(I_n-AA^{+}\right)B=0$
if and only if
$\rk \left(\left[\begin{matrix}A&B\end{matrix}\right]\right)=\rk \left(A\right)$,
we get the following Theorem \ref{suibian-2}.

\begin{theorem}
\label{suibian-2}
Let the DCGI ${\M}^{\mbox{\tiny\textcircled{\#}}}$ of
${\M}=A+\varepsilon B\in{\mathbb{D}}_{n, n}$   exist,
where $\rk(A)=r$.
Then ${\M}^{\mbox{\tiny\textcircled{\#}}}
 =A^{\mbox{\tiny\textcircled{\#}}}-\varepsilon {A^{\mbox{\tiny\textcircled{\#}}}BA^{\mbox{\tiny\textcircled{\#}}}}$
 is equivalent to
${\rm Ind}\left(A\right)=1$ and
$\rk \left(\left[\begin{matrix}A&B\end{matrix}\right]\right)
=\rk \left(A\right)$.
\end{theorem}

Next, we continue to analyze DGGI in the form of
${\M}^{\#}=A^{\#}-\varepsilon {A^{\#}BA^{\#}}$.

\begin{theorem}
\label{gaogaogao5}
Let the DGGI ${\M}^{\#}$ of  ${\M}=A+\varepsilon B\in{\mathbb{D}}_{n, n}$    exist,
where $\rk(A)=r$.
Then ${\M}^{\#}=A^{\#}-\varepsilon {A^{\#}BA^{\#}}$
  is equivalent to
\begin{align}
\label{The-2-M-eq12}
B\left(I_n-AA^{\#}\right)=0
\
\mbox{\rm and  }
\left(I_n-AA^{\#}\right)B=0.
\end{align}
\end{theorem}

\begin{proof}
  Suppose that $A$ is core invertible.
Let   the decomposition of   $A$ be as  in $(\ref{core-1})$,
the decomposition of $A^{\#}$ be as  in (\ref{core-4})
and the form of $B$ be as in (\ref{core-7}).
Then we have
\begin{align}
\nonumber
&\left(A^{\#}\right)^2B(I_n-AA^{\#})=
 U\left[\begin{matrix}
                    0&       \left(\Sigma K\right)^{-2}\left( -B_1K^{-1}L
                    +B_2-K^{-1}LB_3K^{-1}L
                    +K^{-1}LB_4\right)
\\                  0&                           0
\end{matrix}\right]U^\mathrm{T}
\end{align}
and
\begin{align}
\label{xinThe-2-M-eq14}
\left(I_n-AA^{\#}\right)B\left(A^{\#}\right)^2
=
U\left[\begin{matrix}
           -K^{-1}LB_3\left(\Sigma K\right)^{-2}&           -K^{-1}LB_3\left(\Sigma K\right)^{-2}K^{-1}L
\\                 B_3\left(\Sigma K\right)^{-2}&           B_3\left(\Sigma K\right)^{-2}K^{-1}L
\end{matrix}\right]U^\mathrm{ T }.
\end{align}

``$\Longrightarrow$''
\
Let the DGGI ${\M}^{\#}$ of ${\M}\in{\mathbb{D}}_{n, n}$ exist,
then ${\rm Ind}({\M})=1$ from Theorem \ref{core-20}. Therefore, ${\rm Ind}(A)=1$
and $B_4=B_3K^{-1}L$ from Theorem \ref{The-2-40}.
Thus,
\begin{align}
\label{xinThe-2-M-eq13bian}
\left(A^{\#}\right)^2B(I_n-AA^{\#})
&
=
U\left[\begin{matrix}
                    0&        -\left(\Sigma K\right)^{-2}B_1K^{-1}L+\left(\Sigma K\right)^{-2}B_2
\\                  0&                           0
\end{matrix}\right]U^\mathrm{ T }.
\end{align}

If
${\M}^{\#}=A^{\#}-\varepsilon {A^{\#}BA^{\#}}$,
by (\ref{qunni-gongshi}),
 (\ref{xinThe-2-M-eq14}) and
 (\ref{xinThe-2-M-eq13bian}),
 we have $\left(A^{\#}\right)^2B\left(I_n-AA^{\#}\right)=0$
 and
$\left(I_n-AA^{\#}\right)B\left(A^{\#}\right)^2=0$,
 so $B_2=B_1K^{-1}L$ and $B_3=0$.
It follows that
 \begin{align}
 \label{B2}
B=U\left[\begin{matrix}
      B_1&      B_1K^{-1}L
\\      0&            0
\end{matrix}\right]U^\mathrm{ T }.
\end{align}
Therefore,
we have
\begin{align*}
B\left(I_n-AA^{\#}\right)
&
=
U\left[\begin{matrix}
     B_1&   B_1K^{-1}L
\\     0&       0
\end{matrix}\right]U^\mathrm{ T }U
\left[\begin{matrix}
       0&        -K^{-1}L
\\     0&       I_{n-r}
\end{matrix}\right]U^\mathrm{ T }=0;
\\
\left(I_n-AA^{\#}\right)B
&
=
U\left[\begin{matrix}
       0&       -K^{-1}L
\\     0&       I_{n-r}
\end{matrix}\right]U^\mathrm{ T }U
\left[\begin{matrix}
       B_1&   B_1K^{-1}L
\\       0&       0
\end{matrix}\right]U^\mathrm{ T }=0,
\end{align*}
that is, (\ref{The-2-M-eq12}) is established.

``$\Longleftarrow$''
\
Assuming $B\left(I_n-AA^{\#}\right)=0$ and $\left(I_n-AA^{\#}\right)B=0$,
it is easy to check that
 $\left(I_n-AA^{\#}\right)B\left(I_n-A^{\#}A\right)=0$,
$\left(A^{\#}\right)^2B\left(I_n-AA^{\#}\right)=0$
and
$\left(I_n-AA^{\#}\right)B\left(A^{\#}\right)^2=0$.
According to Lemma \ref{core-yingli}, the DGGI of ${\M}$ exists.
Therefore, by using  (\ref{qunni-gongshi}),
 we   get  ${\M}^{\#}=A^{\#}-\varepsilon A^{\#}BA^{\#}$.
\end{proof}

In the following Theorem \ref{xinjia-1-1},
Theorem \ref{xinjia-1-2}
and
Theorem \ref{xinjia-1-3},
we consider the relationships among
${\M}^{\#}=A^{\#}-\varepsilon {A^{\#}BA^{\#}}$,
$ {\M}^{+}=A^{+}-\varepsilon A^{+}BA^{+}$
and
${\M}^{\mbox{\tiny\textcircled{\#}}}
 =A^{\mbox{\tiny\textcircled{\#}}}-\varepsilon {A^{\mbox{\tiny\textcircled{\#}}}BA^{\mbox{\tiny\textcircled{\#}}}}$.
\begin{theorem}
\label{xinjia-1-1}
Let ${\M}=A+\varepsilon B\in{\mathbb{D}}_{n, n}$, then DGGI exists
and ${\M}^{\#}=A^{\#}-\varepsilon A^{\#}BA^{\#}$
if and only if ${\rm Ind}\left(A\right)=1$,
 DMPGI exists   and
${\M}^{+}=A^{+}-\varepsilon A^{+}BA^{+} $.
\end{theorem}

\begin{proof}
Suppose that $A$ is group invertible.
Let   the decomposition of   $A$ be as  in $(\ref{core-1})$,
and the form of $B$ be as in (\ref{core-7}).

``$\Longrightarrow$''
\
Let DGGI exist and ${\M}^{\#}=A^{\#}-\varepsilon A^{\#}BA^{\#}$,
then $B\left(I_n-AA^{\#}\right)=0$ and  $\left(I_n-AA^{\#}\right)B=0$ from Theorem \ref{gaogaogao5}.

Put equations (\ref{core-1}), (\ref{core-4})and (\ref{core-7})
 into $\left(I_n-AA^{\#}\right)B=0$,
then we get
\begin{align}
\nonumber
 \left(I_n-AA^{\#}\right)B=
 U\left[\begin{matrix}
       0&       -K^{-1}L
\\     0&       I_{n-r}
\end{matrix}\right]
\left[\begin{matrix}
       B_1&       B_2
\\     B_3&       B_4
\end{matrix}\right]U^\mathrm{ T }=
U\left[\begin{matrix}
       -K^{-1}LB_3&       -K^{-1}LB_4
\\             B_3&            B_4
\end{matrix}\right]U^\mathrm{ T }=0.
\end{align}
 Therefore,
 $B_3=0$ and $B_4=0$.
By substituting   (\ref{core-1}) and  (\ref{core-4})
into $B\left(I_n-AA^{\#}\right)=0$,
 It follows
that
$$B\left(I_n-AA^{\#}\right)=
U\left[\begin{matrix}
         B_1&       B_2
\\         0&        0
\end{matrix}\right]
\left[\begin{matrix}
       0&       -K^{-1}L
\\     0&       I_{n-r}
\end{matrix}\right]U^\mathrm{ T }=
U\left[\begin{matrix}
       0&       -B_1K^{-1}L+B_2
\\     0&       0
\end{matrix}\right]U^\mathrm{ T }
=0,$$
that is,
$B_2=B_1K^{-1}L$.

Since  $B_3=0$, $B_4=0$ and $B_2=B_1K^{-1}L$,
substituting   (\ref{core-1}) and (\ref{core-3})
into $\left(I_n-AA^{+}\right)B$  and $B\left(I_n-A^{+}A\right)$,
we have
$\left(I_n-AA^{+}\right)B=B\left(I_n-A^{+}A\right)=0$.
By applying Lemma \ref{lemma2}, we know that DMPGI exists and
 ${\M}^{+}=A^{+}-\varepsilon A^{+}BA^{+}.$

And because the DGGI exists, then ${\rm Ind}\left(A\right)=1$.
To sum up, ${\rm Ind}\left(A\right)=1$, DMPGI exists and ${\M}^{+}=A^{+}-\varepsilon A^{+}BA^{+}.$

``$\Longleftarrow$''
\
Let ${\rm Ind}\left(A\right)=1$, DMPGI exist
 and
  ${\M}^{+}=A^{+}-\varepsilon A^{+}BA^{+}.$
By applying Lemma \ref{lemma2}, we know $\left(I_n-AA^{+}\right)B=0$
and $B\left(I_n-A^{+}A\right)=0$.
Substituting    (\ref{core-1}), (\ref{core-3}) and (\ref{core-7})
 into $\left(I_n-AA^{+}\right)B=0$, we get
$$\left(I_n-AA^{+}\right)B=
 U\left[\begin{matrix}
       0&          0
\\     0&       I_{n-r}
\end{matrix}\right]
\left[\begin{matrix}
       B_1&       B_2
\\     B_3&       B_4
\end{matrix}\right]U^\mathrm{ T }=
U\left[\begin{matrix}
                 0&       0
\\             B_3&      B_4
\end{matrix}\right]U^\mathrm{ T }=0,$$
 that is,
 $B_3=0$ and $B_4=0$.

Put   (\ref{core-1}), (\ref{core-3}) and (\ref{core-7})
into $B\left(I_n-A^{+}A\right)=0$,
by applying  $B_3=0$ and $B_4=0$,
we get
\begin{align*}
B\left(I_n-A^{+}A\right)
&=U\left[\begin{matrix}
         B_1&       B_2
\\         0&        0
\end{matrix}\right]
\left[\begin{matrix}
     I_r-K^\mathrm{ T }K&       -K^\mathrm{ T }L
\\    -L^\mathrm{ T }K&       I_{n-r}-L^\mathrm{ T }L
\end{matrix}\right]U^\mathrm{ T }
\\
&=U\left[\begin{matrix}
       B_1-B_1K^\mathrm{ T }K-B_2L^\mathrm{ T }K
        & -B_1K^\mathrm{ T }L+B_2-B_2L^\mathrm{ T }L
\\                                             0&                  0
\end{matrix}\right]U^\mathrm{ T }
=0.
\end{align*}
Therefore,
\begin{subnumcases}{}
 \label{fangchengzu1}
 B_1=B_1K^\mathrm{ T }K+B_2L^\mathrm{ T }K
 =\left(B_1K^\mathrm{ T }+B_2L^\mathrm{ T }\right)K,
 \\
 \label{fangchengzu2}
  B_2=B_2L^\mathrm{ T }L+B_1K^\mathrm{ T }L
  =\left(B_2L^\mathrm{ T }+B_1K^\mathrm{ T }\right)L.
 \end{subnumcases}

Since ${\rm Ind}\left(A\right)=1$,
then $K$ is  nonsingular.
Applying (\ref{fangchengzu1})
gives
\begin{align}
\label{zhongjianchangwu}
B_1K^\mathrm{ T }+B_2L^\mathrm{ T }=B_1K^{-1}.
\end{align}
Substituting  equation (\ref{zhongjianchangwu}) into (\ref{fangchengzu2}),
we have
$B_2=B_1K^{-1}L.$

Since  $B_3=0$, $B_4=0$ and $B_2=B_1K^{-1}L$,
substituting   (\ref{core-1}) and (\ref{core-3})
into
$B\left(I_n-AA^{\#}\right)$ and  $\left(I_n-AA^{\#}\right)B$,
we have
$B\left(I_n-AA^{\#}\right)=0$ and  $\left(I_n-AA^{\#}\right)B=0$.
According to Theorem \ref{gaogaogao5},
DGGI exists and  ${\M}^{\#}=A^{\#}-\varepsilon A^{\#}BA^{\#}$.
\end{proof}

\begin{theorem}
\label{xinjia-1-2}
 Let ${\M}=A+\varepsilon B\in{\mathbb{D}}_{n, n}$.
 If ${\rm Ind}\left(A\right)=1$,
 DMPGI exists and
  ${\M}^{+}
 =A^{+}-\varepsilon A^{+}BA^{+}$,
 then DCGI exists and
 ${\M}^{\mbox{\tiny\textcircled{\#}}}
 =A^{\mbox{\tiny\textcircled{\#}}}-\varepsilon A^{\mbox{\tiny\textcircled{\#}}}BA^{\mbox{\tiny\textcircled{\#}}}$.
\end{theorem}

\begin{proof}
Since ${\rm Ind}\left(A\right)=1$, DMPGI exists
and
${\M}^{+}=A^{+}-\varepsilon A^{+}BA^{+}$,
then $\left(I_n-AA^{+}\right)B=0$  by applying Lemma \ref{lemma2}.

According to Theorem \ref{suibian},  if ${\rm Ind}\left(A\right)=1$, $\left(I_n-AA^{+}\right)B=0$,
then the DCGI ${\M}^{\mbox{\tiny\textcircled{\#}}}$ of  ${\M}$  exists,
and
${\M}^{\mbox{\tiny\textcircled{\#}}}
=A^{\mbox{\tiny\textcircled{\#}}}-\varepsilon {A^{\mbox{\tiny\textcircled{\#}}}BA^{\mbox{\tiny\textcircled{\#}}}}.$
\end{proof}

\begin{theorem}
\label{xinjia-1-3}
Let ${\M}=A+\varepsilon B\in{\mathbb{D}}_{n, n}$.
If ${\M}^{\#}$ exists and ${\M}^{\#}=A^{\#}-\varepsilon A^{\#}BA^{\#}$,
then DCGI exists and
${\M}^{\mbox{\tiny\textcircled{\#}}}=A^{\mbox{\tiny\textcircled{\#}}}-\varepsilon A^{\mbox{\tiny\textcircled{\#}}}BA^{\mbox{\tiny\textcircled{\#}}}$.
\end{theorem}

\begin{proof}
If ${\M}^{\#}$ exists and ${\M}^{\#}=A^{\#}-\varepsilon A^{\#}BA^{\#}$,
then
${\rm Ind}\left(A\right)=1$,
DMPGI exists and ${\M}^{+}=A^{+}-\varepsilon A^{+}BA^{+}$ from Theorem \ref{xinjia-1-1}.
And
 then DCGI exists
and
${\M}^{\mbox{\tiny\textcircled{\#}}}=A^{\mbox{\tiny\textcircled{\#}}}-\varepsilon A^{\mbox{\tiny\textcircled{\#}}}BA^{\mbox{\tiny\textcircled{\#}}}$
from Theorem \ref{xinjia-1-2}.
\end{proof}

  \begin{example}
Let \label{xinjializi-1}
${\M}=
\left[\begin{matrix}
    1&         0
\\  0&         0
\end{matrix}\right]+\varepsilon
\left[\begin{matrix}
    0&    1
\\  0&    0
\end{matrix}\right]$.
By applying
(\ref{The-DMPGI}),
 (\ref{qunni-gongshi})
 and
 (\ref{changjinhua}),
we have
\begin{align}
{\M}^{+}=
\left[\begin{matrix}
    1&         0
\\  0&         0
\end{matrix}\right]+\varepsilon
\left[\begin{matrix}
    0&    0
\\  1&    0
\end{matrix}\right],
\
{\M}^{\#}=
\left[\begin{matrix}
    1&         0
\\  0&         0
\end{matrix}\right]+\varepsilon
\left[\begin{matrix}
    0&    1
\\  0&    0
\end{matrix}\right],
\
{\M}^{\mbox{\tiny\textcircled{\#}}}
 =
\left[\begin{matrix}
    1&         0
\\  0&         0
\end{matrix}\right]+\varepsilon
\left[\begin{matrix}
    0&    0
\\  0&    0
\end{matrix}\right]=
\left[\begin{matrix}
    1&         0
\\  0&         0
\end{matrix}\right],
\nonumber
\end{align}
and
\begin{align}
\nonumber
 A^{\mbox{\tiny\textcircled{\#}}}-\varepsilon A^{\mbox{\tiny\textcircled{\#}}}BA^{\mbox{\tiny\textcircled{\#}}}
 =
\left[\begin{matrix}
    1&         0
\\  0&         0
\end{matrix}\right],
{\M}^{\PP}
=A^{+}-\varepsilon A^{+}BA^{+}
=\left[\begin{matrix}
    1&         0
\\  0&         0
\end{matrix}\right],
A^{\#}-\varepsilon A^{\#}BA^{\#}
=\left[\begin{matrix}
    1&         0
\\  0&         0
\end{matrix}\right].
\end{align}

It is easy to see that
${\M}^{\mbox{\tiny\textcircled{\#}}}
 =A^{\mbox{\tiny\textcircled{\#}}}-\varepsilon A^{\mbox{\tiny\textcircled{\#}}}BA^{\mbox{\tiny\textcircled{\#}}}$,
${\M}^{+}\neq {\M}^{\PP} $
and
${\M}^{\#} \neq  A^{\#}-\varepsilon A^{\#}BA^{\#}$.
 This means that when DCGI exists and ${\M}^{\mbox{\tiny\textcircled{\#}}}
 =A^{\mbox{\tiny\textcircled{\#}}}-\varepsilon A^{\mbox{\tiny\textcircled{\#}}}BA^{\mbox{\tiny\textcircled{\#}}}$,
  there is not necessarily ${\M}^{+}={\M}^{\PP}=A^{+}-\varepsilon A^{+}BA^{+}$
   or
   ${\M}^{\#}= A^{\#}-\varepsilon A^{\#}BA^{\#}.$
\end{example}

\subsection{Symmetric dual matrix}

We know that in real field, the index of a symmetric matrix must be equal to one,
 and its core inverse is equal to its Moore-Penrose inverse and its group inverse.
 But it is not true for some symmetric dual matrices.
 Even some
 symmetric dual matrices  have not  DCGIs and DGGIs.
 For example,
\begin{example}
\label{Ex-2}
${\M}=A+\varepsilon B=
\left[\begin{matrix}
    a&         0&     0
\\  0&         b&     0
\\  0&         0&     0
\end{matrix}\right]+\varepsilon
\left[\begin{matrix}
    0&    0&     0
\\  0&    0&     0
\\  0&    0&     c
\end{matrix}\right] ,$
 where $a, b$ and $c$ are not 0.
Because $\rk \left(A\right)=\rk \left(A^2\right)=2$,
$\rk \left(\left[\begin{matrix}A&B\left(I_3-AA^{\#}\right)\end{matrix}\right]\right)=3$ and $2\neq 3$,
we can get that the dual index of $ \M $ is not one,
 ${\M}$ does not satisfy  Theorem \ref{zhibiaoqunni-xin} and Theorem \ref{core-20}.
Therefore, ${\M}$ has no DGGI and DCGI.
And ${\M}$ also has no DMPGI from Theorem \ref{tianjiatui}.
\end{example}

\begin{theorem}
\label{gaogaogao4}
If the dual matrix ${\M}$ is a symmetric dual matrix, and the dual index of ${\M}$ is one,
 then
\begin{align}
\label{align label-1}
{\M}^{\#}
={\M}^{+}
={\M}^{\mbox{\tiny\textcircled{\#}}}.
\end{align}
\end{theorem}

\begin{proof}
 Let ${\M}=A+\varepsilon B \in{\mathbb{D}}_{n, n}$,
 ${\M}^\mathrm{ T }={\M}$, $\rk\left(A\right)=r$
 and
 the dual index of ${\M}$ be one.
 According to  Theorem \ref{zhibiaoqunni-xin}, Theorem \ref{zhibiaoMPni} and Theorem \ref{core-20},
 the DGGI, DMPGI and DCGI of  $ \M  $
 exist  simultaneously.
Since ${\M}={\M}^\mathrm{T}$,
we get that
$A$ is symmetrical;
$B$ is symmetrical;
$A^2$ is symmetrical
and
 $\left(A^2\right)^{\#}=\left(A^2\right)^{+}=\left(AA^\mathrm{ T }\right)^{+}=\left(A^\mathrm{ T }A\right)^{+}$.
Therefore,
we have
\begin{align*}
A^{+}
=
A^\#,
\ \
A^{+}BA^{+}
&
=
A^{\#}BA^{\#}
\\
\left(A^\mathrm{ T }A\right)^{+}B^\mathrm{ T }\left(I_n-AA^{+}\right)
&
=
\left(A^{\#}\right)^2B\left(I_n-A^{\#}A\right),
\\
\left(I_n-A^{+}A\right)B^\mathrm{ T }\left(AA^\mathrm{ T }\right)^{+}
&
=
\left(I_n-AA^{\#}\right)B\left(A^{\#}\right)^2.
\end{align*}
Then
by applying (\ref{The-DMPGI}) and
 (\ref{qunni-gongshi}),
we derive that ${\M}^{\#}={\M}^{+}$.

It follows from ${\M}^{\#}={\M}^{+}$ and ${\M}^{\mbox{\tiny\textcircled{\#}}}={\M}^{\#}{\M}{\M}^{+}$ in (\ref{jinjinhua}) that
${\M}^{\mbox{\tiny\textcircled{\#}}}={\M}^{\#}{\M}{\M}^{+}
={\M}^{\#}{\M}{\M}^{\#}={\M}^{\#}={\M}^{+}$.
Therefore,
${\M}^{\#}
={\M}^{+}
={\M}^{\mbox{\tiny\textcircled{\#}}}$.
\end{proof}

\begin{theorem}
\label{youhefang}
If the DCGI ${\M}^{\mbox{\tiny\textcircled{\#}}}=A^{\mbox{\tiny\textcircled{\#}}}+\varepsilon R$
 of   ${\M}=A+\varepsilon B\in  {\mathbb{D}}_{n, n}$ exists,
and ${\M}$ is a symmetric dual matrix, then
\begin{align}
\label{shishi}
{\M}^{\mbox{\tiny\textcircled{\#}}}
=
\left({\M}^{\mbox{\tiny\textcircled{\#}}}\right)^\mathrm{ T }
=
A^{\mbox{\tiny\textcircled{\#}}}
+\varepsilon\left(\left(A^{\mbox{\tiny\textcircled{\#}}}\right)^{2}B\left(I-AA^{\mbox{\tiny\textcircled{\#}}}\right)
+\left(I-AA^{\mbox{\tiny\textcircled{\#}}}\right)B\left(A^{\mbox{\tiny\textcircled{\#}}}\right)^{2}\right).
\end{align}
\end{theorem}

\begin{proof}

According to Theorem \ref{zhibiaoqunni-xin}, Theorem \ref{zhibiaoMPni} and Theorem \ref{core-20},
if the DCGI of ${\M}$ exists, then DGGI and DMPGI exist.

And then ${\M}^{\mbox{\tiny\textcircled{\#}}}
={\M}^{\#}{\M}{\M}^{+}={\M}^{+}{\M}{\M}^{+}={\M}^{+}$ by   (\ref{align label-1}).
 From
 $\left({\M}^{+}\right)^\mathrm{ T }={\M}^{+}$ (see \cite{Udwadia2021mamt156}),
so $\left({\M}^{\mbox{\tiny\textcircled{\#}}}\right)^\mathrm{ T }={\M}^{\mbox{\tiny\textcircled{\#}}}.$

Since ${\M}$ is a symmetric dual matrix, then
we have $A^\mathrm{ T }=A$ and $ B^\mathrm{ T }=B$.
Moreover,
$A^{\mbox{\tiny\textcircled{\#}}}=A^{\#}=A^+=(A^{\mbox{\tiny\textcircled{\#}}})^\mathrm{ T }
 =(A^{\#})^\mathrm{ T }=(A^+)^\mathrm{ T }.$

According to the formula (\ref{changjinhua}),  we can get
\begin{align*}
{\M}^{\mbox{\tiny\textcircled{\#}}}&
=A^{\mbox{\tiny\textcircled{\#}}}+
\varepsilon\left(-A^{\mbox{\tiny\textcircled{\#}}}BA^{+}+A^{\#}BA^{+}
-A^{\#}BA^{\mbox{\tiny\textcircled{\#}}}+A^{\mbox{\tiny\textcircled{\#}}}(BA^{+})^\mathrm{ T }\left(I-AA^{+}\right)+\left(I-AA^{\#}\right)BA^{\#}A^{\mbox{\tiny\textcircled{\#}}}\right)
\\
&=A^{\mbox{\tiny\textcircled{\#}}}+\varepsilon\left(-A^{\mbox{\tiny\textcircled{\#}}}BA^{\mbox{\tiny\textcircled{\#}}}
+A^{\mbox{\tiny\textcircled{\#}}}\left(BA^{+}\right)^\mathrm{ T }\left(I-AA^{+}\right)
+\left(I-AA^{\#}\right)BA^{\#}A^{\mbox{\tiny\textcircled{\#}}}\right)
\\
&=A^{\mbox{\tiny\textcircled{\#}}}+\varepsilon\left(-A^{\mbox{\tiny\textcircled{\#}}}BA^{\mbox{\tiny\textcircled{\#}}}
+A^{\mbox{\tiny\textcircled{\#}}}\left(BA^{\mbox{\tiny\textcircled{\#}}}\right)^\mathrm{ T }(I-AA^{\mbox{\tiny\textcircled{\#}}})+\left(I-AA^{\mbox{\tiny\textcircled{\#}}}\right)BA^{\mbox{\tiny\textcircled{\#}}}A^{\mbox{\tiny\textcircled{\#}}}\right)
\\
&=A^{\mbox{\tiny\textcircled{\#}}}+\varepsilon\left(-A^{\mbox{\tiny\textcircled{\#}}}BA^{\mbox{\tiny\textcircled{\#}}}
+A^{\mbox{\tiny\textcircled{\#}}}\left(A^{\mbox{\tiny\textcircled{\#}}}\right)^\mathrm{ T }B^\mathrm{ T }\left(I-AA^{\mbox{\tiny\textcircled{\#}}}\right)
+\left(I-AA^{\mbox{\tiny\textcircled{\#}}}\right)BA^{\mbox{\tiny\textcircled{\#}}}A^{\mbox{\tiny\textcircled{\#}}}\right)
\\
&=A^{\mbox{\tiny\textcircled{\#}}}+\varepsilon\left(-A^{\mbox{\tiny\textcircled{\#}}}BA^{\mbox{\tiny\textcircled{\#}}}
+\left(A^{\mbox{\tiny\textcircled{\#}}}\right)^{2}B \left(I-AA^{\mbox{\tiny\textcircled{\#}}}\right)
+\left(I-AA^{\mbox{\tiny\textcircled{\#}}}\right)B\left(A^{\mbox{\tiny\textcircled{\#}}}\right)^{2}\right),
\end{align*}
that is, (\ref{shishi}).
\end{proof}

Next, we consider  some properties   of DCGI, DGGI and DMPGI
in special forms when the research object is a symmetric dual matrix.

\begin{lemma}[\cite{Marsaglia1974LMA269}]
\label{lemmalemma}
 Let $A \in{\mathbb{R}}_{n, n}$ and $B\in{\mathbb{R}}_{n, n}$, then
 $\mathcal{R}\left(B\right)\subseteq \mathcal{R}\left(A\right)$
  if and only if
  $\rk \left(\left[\begin{matrix}A&B\end{matrix}\right]\right)=\rk \left(A\right).$
\end{lemma}

\begin{theorem}
\label{gaogaogaogao6}
Let ${\M}=A+\varepsilon B$ be an $n$-order symmetric dual matrix.
Then the following conditions are equivalent:
\begin{description}
  \item[(1)]
The DCGI ${\M}^{\mbox{\tiny\textcircled{\#}}}$ of
 ${\M} $ exists,
 and
  ${\M}^{\mbox{\tiny\textcircled{\#}}}={\M}^{\#}={\M}^{+}
=
{\M}^{\PP}
=
A^{+}-\varepsilon A^{+}BA^{+}$;
  \item[(2)]
  $\left(I _n-AA^{+}\right)B=0$
  or $\left(I _n-AA^{\#}\right)B=0$
  or $\left(I _n-AA^{\mbox{\tiny\textcircled{\#}}}\right)B=0$;
  \item[(3)]
 $\rk \left(\left[\begin{matrix}A&B\end{matrix}\right]\right)=\rk \left(A\right);$
  \item[(4)]
 $\mathcal{R}\left(B\right)\subseteq \mathcal{R}\left(A\right)$;
  \item[(5)]
  $B\left(I _n-AA^{+}\right)=0$
  or $B\left(I _n-AA^{\#}\right)=0$
  or $B\left(I _n-AA^{\mbox{\tiny\textcircled{\#}}}\right)=0$.
  \end{description}
\end{theorem}

\begin{proof}

 Let ${\M}\in{\mathbb{D}}_{n, n}$,  $\rk\left(A\right)=r$ and  ${\M}^T= {\M}$.
  Since ${\M}$ is a symmetric dual matrix,
  then the real part $A$ is symmetric,
and
$A^{\mbox{\tiny\textcircled{\#}}}=A^{+}=A^{\#}$.
Therefore, (1) is ${\M}^{\mbox{\tiny\textcircled{\#}}}
={\M}^{\#}={\M}^{+}={\M}^{\PP}
=A^{+}-\varepsilon A^{+}BA^{+}=A^{\#}-\varepsilon A^{\#}BA^{\#}
=A^{\mbox{\tiny\textcircled{\#}}}-\varepsilon A^{\mbox{\tiny\textcircled{\#}}}BA^{\mbox{\tiny\textcircled{\#}}}$.
And (2) , (5) are equivalent.

``$(1)\Longrightarrow(2)$''
\
 Let ${\M}^{\mbox{\tiny\textcircled{\#}}}={\M}^{\#}={\M}^{+}=
{\M}^{\PP}=A^{\#}-\varepsilon A^{\#}BA^{\#}$,
from Theorem \ref{gaogaogao5},
 $B\left(I_n-AA^{\#}\right)=0$ and $\left(I_n-AA^{\#}\right)B=0$.
 Therefore, we get $\left(I_n-AA^{\#}\right)B=0$.
 And because $A^{\mbox{\tiny\textcircled{\#}}}=A^{+}=A^{\#}$,
 (2) is established.

``$(1)\Longleftarrow(2)$''
\
Let $\left(I _n-AA^{\#}\right)B=0$.
Because ${\M}$ is symmetrical,
 $B\left(I _n-AA^{\#}\right)=0$.
Applying Theorem \ref{gaogaogao5} gives that DGGI ${\M}^{\#}$ of  ${\M}=A+\varepsilon B$ exists
 and ${\M}^{\#}=A^{\#}-\varepsilon {A^{\#}BA^{\#}}$.

 According to Theorem \ref{zhibiaoqunni-xin},
 Theorem \ref{zhibiaoMPni} and Theorem \ref{core-20},
 when ${\M}^{\#}$ exists,
 ${\M}^{\mbox{\tiny\textcircled{\#}}}$
 and ${\M}^{+}$ exist,
 and the dual index of ${\M}$ is 1.
According to Theorem \ref{gaogaogao4},
${\M}^{\mbox{\tiny\textcircled{\#}}}={\M}^{+}={\M}^{\#}$.
Then ${\M}^{\mbox{\tiny\textcircled{\#}}}={\M}^{+}
={\M}^{\#}=A^{\#}-\varepsilon {A^{\#}BA^{\#}}$.
Considering that ${\M}$ is a symmetric dual matrix,
then $A^{\mbox{\tiny\textcircled{\#}}}=A^{+}=A^{\#}$.
Therefore, (1) is established.

According to Theorem \ref{suibian-2}, (1) and (3) are equivalent.
According to Lemma \ref{lemmalemma}, (3) and (4) are equivalent.
\end{proof}
According to Theorem \ref{gaogaogao4} and Theorem \ref{youhefang},
 it is easy to get the following Corollary \ref{zuihouyige}.

\begin{corollary}
\label{zuihouyige}
Let ${\M}=A+\varepsilon B\in  {\mathbb{D}}_{n, n}$ be a symmetric dual matrix.
If the DCGI, DGGI, DMPGI  of ${\M}$ exist, then  they are symmetrical and equal.
\end{corollary}

\section{Applications of DCGI in Linear Dual Equations}

In this section, we use two  examples to illustrate some application of DCGI
 in  solving linear dual equations.

First we consider solving a consistent linear dual equation  by DCGI in Example \ref{example-3.1}.
    We give a general solution to the consistent dual equation.

\begin{example}
\label{example-3.1}
Let ${\M}\widehat{x}=\widehat{b}$ be a consistent equation,
where
\begin{align}
\nonumber
{\M}=A+\varepsilon B=
\left[\begin{matrix}
       1&      0
\\     0&      0
\end{matrix}\right]+\varepsilon
\left[\begin{matrix}
       1&      1
\\     1&      0
\end{matrix}\right]=
\left[\begin{matrix}
       1+\varepsilon&      \varepsilon
\\       \varepsilon&          0
\end{matrix}\right],\end{align}
$\widehat{b}=
\left[\begin{matrix}
       1
\\     0
\end{matrix}\right]+
\varepsilon\left[\begin{matrix}
       1
\\     1
\end{matrix}\right]=
\left[\begin{matrix}
       1+\varepsilon
\\       \varepsilon
\end{matrix}\right]$
and
$\widehat{x}=
\left[\begin{matrix}
      \widehat{x_1}
\\    \widehat{x_2}
\end{matrix}\right], \widehat{x_i}=x_i+\varepsilon x_i^{'}, i=1,2$.

It is easy to check that
$\rk \left(A^{2}\right)
=\rk \left(A\right)
=\rk\left (\left[\begin{matrix}A&B(I_n-AA^{\#}\end{matrix}\right]\right)=1$.
By applying Theorem \ref{The-2-4},
we get that
the dual index of ${\M}$ is one,
that is,
the DCGI ${\M}^{\mbox{\tiny\textcircled{\#}}}$  exists.
Applying $(\ref{changjinhua})$ gives
\begin{align}
\label{shishi}
{\M}^{\mbox{\tiny\textcircled{\#}}}=G+\varepsilon R=
\left[\begin{matrix}
       1&      0
\\     0&      0
\end{matrix}\right]+\varepsilon
\left[\begin{matrix}
       -1&      1
\\      1&      0
\end{matrix}\right]=
\left[\begin{matrix}
       1-\varepsilon&      \varepsilon
\\       \varepsilon&          0
\end{matrix}\right].
\end{align}
Thus,
 $${\M}^{\mbox{\tiny\textcircled{\#}}}\widehat{b}=
\left[\begin{matrix}
       1-\varepsilon&      \varepsilon
\\       \varepsilon&          0
\end{matrix}\right]
\left[\begin{matrix}
       1+\varepsilon
\\       \varepsilon
\end{matrix}\right]=
\left[\begin{matrix}
               1
\\       \varepsilon
\end{matrix}\right].$$

\

Furthermore,
let
\begin{align}
\label{lizi-1}
\widehat{x}
=
{\M}^{\mbox{\tiny\textcircled{\#}}}\widehat{b}+\left(I-{\M}^{\mbox{\tiny\textcircled{\#}}}{\M}\right)\widehat{w}
=
\left[\begin{matrix}
              1
\\       \varepsilon
\end{matrix}\right]+
\left[\begin{matrix}
                 0 &      -\varepsilon
\\       -\varepsilon&          1
\end{matrix}\right]\widehat{w},
\end{align}
where $\widehat{w}$ is an arbitary n-by-1 dual column vector.

By substituting  $(\ref{lizi-1})$ into ${\M}\widehat{x}=\widehat{b}$,
we can get
\begin{align*}
{\M}\widehat{x}-\widehat{b}
&
=
{\M}({\M}^{\mbox{\tiny\textcircled{\#}}}\widehat{b}
+\left(I-{\M}^{\mbox{\tiny\textcircled{\#}}}{\M}\right)\widehat{w})
-\widehat{b}
=
{\M}{\M}^{\mbox{\tiny\textcircled{\#}}}\widehat{b}-\widehat{b}
+\left({\M}-{\M}{\M}^{\mbox{\tiny\textcircled{\#}}}{\M}\right)\widehat{w}
\\
&
=
\left[\begin{matrix}
       1+\varepsilon&      \varepsilon
\\       \varepsilon&          0
\end{matrix}\right]
\left[\begin{matrix}
              1
\\       \varepsilon
\end{matrix}\right]-\widehat{b}
=
\left[\begin{matrix}
       1+\varepsilon
\\      \varepsilon
\end{matrix}\right]-\widehat{b}
=
\left[\begin{matrix}
       1+\varepsilon
\\      \varepsilon
\end{matrix}\right]-
\left[\begin{matrix}
       1+\varepsilon
\\      \varepsilon
\end{matrix}\right]
=0,
\end{align*}
that is,
$(\ref{lizi-1})$ is the solution to ${\M}\widehat{x}=\widehat{b}$.

Meanwhile,
let $\widehat{x}$ be any solution to ${\M}\widehat{x}=\widehat{b}$.
Pre-multiplying ${\M}\widehat{x}=\widehat{b}$ by ${\M}^{\mbox{\tiny\textcircled{\#}}}$, we get
${\M}^{\mbox{\tiny\textcircled{\#}}}{\M}\widehat{x}={\M}^{\mbox{\tiny\textcircled{\#}}}\widehat{b}$,
then
 \begin{align}
 \nonumber
 \widehat{x}
 ={\M}^{\mbox{\tiny\textcircled{\#}}}\widehat{b}+\widehat{x}-{\M}^{\mbox{\tiny\textcircled{\#}}}{\M}\widehat{x}
 ={\M}^{\mbox{\tiny\textcircled{\#}}}\widehat{b}+\left(I-{\M}^{\mbox{\tiny\textcircled{\#}}}{\M}\right)\widehat{x}
 =\left[\begin{matrix}
             1
\\       \varepsilon
\end{matrix}\right]+
\left[\begin{matrix}
                  0 &      -\varepsilon
\\       -\varepsilon&          1
\end{matrix}\right]\widehat{x}.
\end{align}
Therefore,
each solution to ${\M}\widehat{x}=\widehat{b}$
 can be written as $(\ref{lizi-1})$ in which  $\widehat{w}=\widehat{x}$.

To sum up,
 (\ref{lizi-1}) is the general solution to ${\M}\widehat{x}=\widehat{b}$.
\end{example}

In order to solve inconsistent dual linear equation,
Udwadia \cite{Udwadia2021mamt156}
introduces the norm of the dual vector.
Consider the $m$-by-$1$ dual vector $\widehat{u_{i}}=p_{i}+\varepsilon q_{i}$.
Write
\begin{align}
\label{fanshu}
\Vert \widehat{u_{i}} \Vert^{2}=
\left(p_{i}+\varepsilon q_{i}\right)^\mathrm{ T }\left(p_{i}+\varepsilon q_{i}\right)
=\Vert {p_{i}} \Vert^{2}+2\varepsilon p_{i}^\mathrm{ T }q_{i},
\end{align}
where $p_{i}\neq 0$ and $\Vert {p_{i}} \Vert^{2}=p_{i}^\mathrm{ T }{p_{i}}$.
By using the right-most expression in  (\ref{fanshu}),
one  norm of the dual vector $\widehat{u_{i}}$ is given as
\begin{align}
\langle \widehat{u_{i}} \rangle := \Vert p_i\Vert+\Vert q_i\Vert.
\end{align}
%
In  (\ref{fanshu}),
  the dual norm is used to determine the magnitude of the error.
In \cite{Udwadia2021mamt156}, Udwadia introduce the analog of the least-squares solution of any inconsistent dual equation
 ${\M}\widehat{x}=\widehat{b}$ and gives the corresponding solution - analog of the least-squares solution
$\widehat{x}={\M}^{\left(1,3\right)}\widehat{b}+\left(I-{\M}^{\left(1,3\right)}{\M}\right)\widehat{h}$, where $\widehat{h}$ is an arbitrary dual column vector and ${\M}^{\left(1,3\right)}$ exists.
It can be found that the real part $x$ of
the analog of the least-squares solution is the least-squares solution to equation $Ax = b,$
 where $A$, $b$ and $x$ are matched
 with the real parts of ${\M}$, $\widehat{b}$ and $\widehat{x}$ respectively.

\begin{example}
\label{example-3.2}
Let the inconsistent equation be ${\M}\widehat{x}=\widehat{b}$, where
${\M}=A+\varepsilon B=
\left[\begin{matrix}
       4&      2
\\     2&      1
\end{matrix}\right]+\varepsilon
\left[\begin{matrix}
       10&      10
\\      9&      7
\end{matrix}\right]=
\left[\begin{matrix}
       4+10\varepsilon&      2+10\varepsilon
\\      2+9\varepsilon&       1+7\varepsilon
\end{matrix}\right],$
$\widehat{b}=
\left[\begin{matrix}
       0
\\     1
\end{matrix}\right]+
\varepsilon\left[\begin{matrix}
       1
\\     0
\end{matrix}\right]=
\left[\begin{matrix}
       \varepsilon
\\     1
\end{matrix}\right]
$
and
$\widehat{x}=
\left[\begin{matrix}
      \widehat{x_1}
\\    \widehat{x_2}
\end{matrix}\right], \widehat{x_i}=x_i+\varepsilon x_i^{'}, i=1,2$,
that is,
\begin{align}
\label{20220525-4}
\left[\begin{matrix}
        4+10\varepsilon&      2+10\varepsilon
\\      2+9\varepsilon&       1+7\varepsilon
\end{matrix}\right]
\left[\begin{matrix}
       \widehat{x_1}
\\     \widehat{x_2}
\end{matrix}\right]=
\left[\begin{matrix}
        \varepsilon
\\           1
\end{matrix}\right].
\end{align}
Then
\begin{align}
\nonumber
{\M}^{+}=
\left[\begin{matrix}
       0.1600-0.6880\varepsilon&      0.0800-0.1840\varepsilon
\\     0.0800-0.1440\varepsilon&      0.0400+0.0080\varepsilon
\end{matrix}\right].
\end{align}
By applying the Result 12 of  Udwadia in \cite{Udwadia2021mamt156},
the analogue of the least-squares of the inconsistent equation is
\begin{align}
\nonumber
 \widehat{x}&={\M}^{+}\widehat{b}+\left(I-{\M}^{+}{\M}\right)\widehat{h}
 \\
\nonumber
 &=
\left[\begin{matrix}
       0.0800-0.0240\varepsilon
\\     0.0400+0.0880\varepsilon
\end{matrix}\right]+
\left[\begin{matrix}
       0.2000+0.8000\varepsilon&     -0.4000-0.6000\varepsilon
\\    -0.4000-0.8000\varepsilon&      0.8000-0.8000\varepsilon
\end{matrix}\right]\widehat{h},
 \end{align}
 where $\widehat{h}$ is an arbitrary $n$-by-$1$ dual vector.
The norm of the error
\begin{align}
\label{20220525-1}
\langle \widehat{e^{*}} \rangle
=\langle {\M}\widehat{x}-\widehat{b} \rangle
=\langle {\M}{\M}^{+}\widehat{b}-\widehat{b} \rangle
 = \langle \left[\begin{matrix}
        0.4000
\\     -0.8000
\end{matrix}\right] \rangle+\varepsilon\langle \left[\begin{matrix}       0.2800\\     1.0400\end{matrix}\right]\rangle
=1.9715 .
 \end{align}

According to Theorem \ref{The-2-4},
we have
$\rk (A^{2})=\rk (A)
=\rk\left (\left[\begin{matrix}A&B(I_n-AA^{\#}\end{matrix}\right]\right)=1$,
that is,
the dual index of ${\M}$ is one.
 And from Theorem \ref{core-20},
 it can be seen that ${\M}^{\mbox{\tiny\textcircled{\#}}}$ exists.
 Then
\begin{align*}
{\M}^{\mbox{\tiny\textcircled{\#}}}
=G+\varepsilon R
=
\left[\begin{matrix}
       0.1600-0.6720\varepsilon&      0.0800-0.1760\varepsilon
\\     0.0800-0.1760\varepsilon&      0.0400-0.0080\varepsilon
\end{matrix}\right]
\end{align*}
 by the compact formula $(\ref{changjinhua})$.
Denote
\begin{align}
\label{20220525-3}
\widehat{x}
&=
{\M}^{\mbox{\tiny\textcircled{\#}}}\widehat{b}+\left(I-{\M}^{\mbox{\tiny\textcircled{\#}}}{\M}\right)\widehat{w}
\\
\nonumber
&=
\left[\begin{matrix}
       0.0800-0.0160\varepsilon
\\     0.0400+0.0720\varepsilon
\end{matrix}\right]+
\left[\begin{matrix}
       0.2000+0.7200\varepsilon&     -0.4000-0.6400\varepsilon
\\    -0.4000-0.4400\varepsilon&      0.8000-0.7200\varepsilon
\end{matrix}\right]\widehat{w},
\end{align}
 where $\widehat{w}$ is an arbitrary n-by-1 dual vector.
 The norm of the error is
\begin{align}
\nonumber
\langle \widehat{e} \rangle
&=\langle {\M}\widehat{x}-\widehat{b} \rangle
=\langle {\M}{\M}^{\mbox{\tiny\textcircled{\#}}}\widehat{b}-\widehat{b} \rangle
 =\langle \widehat{u_2} \rangle=\left\| m_2\right\|+\left\| n_2\right\|
 \\
 \label{20220525-2}
 & =\left\|  \left[\begin{matrix}
        0.4000
\\     -0.8000
\end{matrix}\right] \right\| +
\left\|  \left[\begin{matrix}
       0.2800
\\     1.0400
\end{matrix}\right] \right\|=1.9715.
\end{align}

Therefore,
from (\ref{20220525-1})  and  (\ref{20220525-2}),
we see that   $\langle \widehat{e^{*}} \rangle=\langle \widehat{e} \rangle=1.9715$,
that is,
(\ref{20220525-3})
is also the analog of the least-squares solution  of
(\ref{20220525-4}).

\end{example}

The two examples in this section calculate the DCGIs of the two dual matrices through the compact formula (\ref{changjinhua}).
When the dual index of any dual matrix is one, its DCGI exists.
On this basis, we can obtain DCGI directly through the compact formula (\ref{changjinhua}).
However, in order to reduce the amount of calculation, we can first consider equation (\ref{core-jianhuan}) in Theorem \ref{suibian}.
If the dual matrix satisfies equation (\ref{core-jianhuan}), then  ${\M}^{\mbox{\tiny\textcircled{\#}}}
=A^{\mbox{\tiny\textcircled{\#}}}-\varepsilon {A^{\mbox{\tiny\textcircled{\#}}}BA^{\mbox{\tiny\textcircled{\#}}}}$.
 Otherwise, we have to use the the compact formula (\ref{changjinhua}).

\section{Conclusions}
The first part of this paper provides some new findings about dual index one of dual matrices, including characterizations of the dual index one. We obtain that DGGI exists if and only if the dual index is one.  Also, when the dual index is one, DMPGI exists and the real part index of the dual Moore-Penrose generalized invertible matrix is one, and vice versa.
 The second part of this paper explores DCGI systematically. Some results from the second part of the paper are as follows:

1.	If a dual core generalized inverse (DCGI) of a dual matrix exists, it is unique.

2.	If DCGI exists, a compact formula for DCGI is given.

3.	We provide a series of equivalent characterizations of the existence of DCGI, for example,
 the dual index is one if and only if DCGI exists.

4.	Relations among MPDGI, DMPGI, DCGI and DGGI are proved.

In the third part, DCGI is applied to linear dual equations through a consistent dual equation and an inconsistent dual equation.

%

\section*{Disclosure statement}
No potential conflict of interest was reported by the authors.



\end{document}